\documentclass[12pt, a4paper]{amsart}

\usepackage[hmargin=30mm, vmargin=25mm, includefoot, twoside]{geometry}
\usepackage[bookmarksopen=true]{hyperref}
\usepackage{amsmath,amsfonts,amsthm,amssymb,dsfont,mathdots}
\usepackage[alphabetic]{amsrefs}
\usepackage{hyperref}

\usepackage{eucal}

\usepackage{times}
\usepackage{wasysym,stmaryrd}

\usepackage{enumerate}

\usepackage{mathrsfs}
\usepackage{enumerate, xspace}
\usepackage{graphicx}
\usepackage{color}

\usepackage{pictexwd,dcpic,epsf}

\newtheorem{lem}{Lemma}[section]
\newtheorem{tw}[lem]{Theorem}
\newtheorem*{mtw}{Main Theorem}

\newtheorem{cor}{Corollary}
\newtheorem{prop}[lem]{Proposition}

\newcommand {\mr}{\mathrm}

\newcommand {\wt}{\widetilde}

\newcommand {\ov}{\overline}

\newcommand {\rems}{\noindent {\bf Remarks. }}

\theoremstyle{definition}
\newtheorem{de}[lem]{Definition}
\newtheorem*{open}{Open problem}

\newcommand {\xz}{X^{(0)}}
\newcommand {\xj}{X^{(1)}}
\newcommand {\dw}{d_{\mathcal W}}

\newcommand {\ag}{A(\gamma)}
\newcommand {\gi}{\mr{girth}}
\newcommand {\di}{\mr{diam}}
\newcommand {\bds}{($\beta,\delta$)--se\-pa\-ration }

\theoremstyle{remark}

\begin{document}

\title[Graphical small cancellation groups with the Haagerup property]{Graphical small cancellation groups with the Haagerup property}

\author{Goulnara Arzhantseva}
\address{Universit\"at Wien, Fakult\"at f\"ur Mathematik\\
Oskar-Morgenstern-Platz 1, 1090 Wien, Austria.
}
\address{Erwin Schr\"odinger International Institute for Mathematical Physics\\
Boltzmanngasse~9, 1090 Wien, Austria.
}
\email{goulnara.arzhantseva@univie.ac.at}

\author{Damian Osajda}
\address{Instytut Matematyczny,
Uniwersytet Wroc\l awski (\textup{on leave})\\
pl.\ Grunwaldzki 2/4,
50--384 Wroc{\l}aw, Poland}
\address{Universit\"at Wien, Fakult\"at f\"ur Mathematik\\
Oskar-Morgenstern-Platz 1, 1090 Wien, Austria.
}
\email{damian.osajda@univie.ac.at}

\thanks{G.A.\ was partially supported by the ERC grant ANALYTIC no.\ 259527.
D.O.\ was partially supported by Narodowe Centrum Nauki, decision no DEC-2012/06/A/ST1/00259.}

\begin{abstract}
We prove the Haagerup property (= Gromov's a-T-menability)  for finitely
generated groups defined by infinite presentations satisfying the graphical $C'(\ov {\lambda})$--small cancellation
condition with respect to graphs endowed with a compatible wall structure. We deduce that these groups are coarsely embeddable into a Hilbert space
and that the strong Baum-Connes conjecture and, hence, the Baum-Connes conjecture with arbitrary coefficients
hold for them.  As the main step we show that $C'(\ov {\lambda})$--complexes satisfy the linear separation property.
Our result provides many new examples and a general technique to show the
Haagerup property for graphical small cancellation groups.
\end{abstract}

\maketitle

\section{Introduction}
\label{s:intro}
The aim of this paper is two-fold: to show a general analytic result, the Haagerup property, for a wide class of graphically presented groups
and to provide an approach to a long-standing problem on the existence of a coarsely embeddable, into a Hilbert space, but not
coarsely amenable group.

Graphical presentations are group presentations where relators are words labeling cycles of a given graph.
Every group has a graphical presentation in a trivial way: the corresponding graph is the disjoint union of simple cycles labeled by the relator words.  
In general, given a labeled graph, one expects that its combinatorial structure and properties of the labeling encode  algebraic and geometric features  of the group it defines. 

The first concrete use of non-trivial graphical presentations is due to Rips and Segev 
in the context of Kaplansky's zero-divisor conjecture:  they give the first example of a torsion-free group that is not a unique product group~\cite{RS}. 
A recent application is the celebrated construction of Gromov's monster group which contains an infinite expander family of graphs in the Cayley
graph~\cites{Gro, AD}.  Both constructions follow a general idea: find an appropriate combinatorial interpretation (e.g.\ the expansion) of
a required group property  and combine it with a suitable small cancellation condition on the labeling 
(e.g.\ the geometric small cancellation in the case of Gromov's monster) in order to produce a group containing a given graph in its Cayley graph. This approach leads 
to spectacular counterexamples: to the unique product property~\cite{RS} and to the Baum-Connes conjecture with arbitrary coefficients~\cite{HLS}.

In contrast to such specific counterexamples, our first goal is to prove an affirmative result, the Haagerup property, for many groups given by 
graphical small cancellation presentations. 

A second countable, locally compact group $G$ has the \emph{Haagerup property}  (or $G$ is \emph{a-T-menable}  in the sense of Gromov) if it possesses
a proper continuous affine isometric action on a Hilbert space. The concept first appeared in the seminal paper of Haagerup \cite{Haa}, where this  property was shown for finitely generated free groups. 
Regarded as a weakening of von Neumann's amenability and a strong negation of Kazhdan's property (T), the Haagerup property has been  revealed independently in harmonic analysis, non-commutative geometry, and ergodic theory \cites{AW,Cho,BoJaS,BR}, \cite{Gro88}*{4.5.C}, \cite{Gro93}*{7.A and 7.E}.
A major breakthrough was a spectacular proof of Higson and Kasparov \cite{HK}  of the strong Baum-Connes conjecture (which is
strictly stronger than the Baum-Connes conjecture with coefficients~\cite{MeyerNest})  
for all groups with the Haagerup property. It follows that the Novikov higher signature conjecture and, for discrete torsion-free groups, the Kadison-Kaplansky idempotents conjecture hold for
these groups. Nowadays, many groups have been shown to have the Haagerup property and  significant applications in K-theory and topology have been discovered \cites{ChCJJV,MislinValette}, making groups with the Haagerup property increasingly fundamental to study.

Groups given by classical small cancellation 
presentations\footnote{The graph is the disjoint union of simple cycles and  the labeling satisfies the classical $C'(1/6)$--small cancellation condition~\cite{LS}.} are known 
to possess the Haagerup property by a result of Wise~\cite{W-sc} for finitely presented groups and by a recent result of the authors~\cite{AO} for infinitely presented groups.
In contrast, there exist non-trivial \emph{graphical small cancellation} presentations defining infinite groups with 
Kazhdan's property (T), hence,  without the Haagerup property~\cites{Gro,Silberman,OllivierWise}.

In this paper,  we determine a natural compatibility assumption between the graph combinatorics and  
the graphical small cancellation condition on the labeling, which guarantee the Haagerup property of the resulting graphically presented group.

\begin{mtw}
\label{t:main}
Let $G$ be a finitely generated group given by graphical presentation satisfying the lacunary walling condition (see Definition~\ref{d:lwpres}).
Then $G$ acts properly on a space with walls. In particular,  $G$ has the Haagerup property.
\end{mtw}

A systematic discussion of groups having proper actions on a space with walls can be found in~\cite{Cor}.
In the terminology of that paper, our main result means that $G$ has Property PW.

We actually prove a stronger result which, in addition, holds for more general spaces.

\begin{tw}\label{t:linear}
A complex $X$ with the lacunary walling condition satisfies the linear separation property, that is, the path metric on $X^{(0)}$ and the wall pseudo-metric are bi-Lipschitz equivalent.
\end{tw}

Our method arises in ambition to approach the following -- still open -- well-known problem (see e.g.~\cite{claire}, \cite{NowakYu}*{5.3.3}):\smallskip

\noindent
\emph{Does there exist a finitely generated group which coarsely embeds into a Hilbert space but which is not coarsely amenable?}\smallskip

The concept of coarse embedding was introduced by Gromov~\cite{Gro93}*{p.218}.
Yu~\cite{Yu}*{Theorem 1.1} proved the coarse Baum-Connes conjecture for every discrete space with bounded geometry coarsely embeddable into a Hilbert space. 
This implies the Novikov higher signature conjecture for all closed manifolds whose fundamental group, viewed with the word length metric, 
admits such a coarse embedding. This result  generated an intense study of groups and metric spaces coarsely embeddable into a Hilbert space.

Coarse amenability is a weak form of amenability. It was introduced in~\cite{Yu}, under the term Property A,  as a sufficient condition for coarse embeddings into a Hilbert space.
For a countable discrete group $G$, coarse amenability is equivalent to the existence of a topological amenable action of $G$ on a compact Hausdorff space and to the $C^*$--exactness of the reduced 
$C^*$--algebra $C^*_r(G)$, see e.g.~\cite{NowakYu}.

All finitely generated groups currently known to be coarsely embeddable into a Hilbert space are, moreover, coarsely amenable. That is,
the above question remains open. Towards a positive answer, we have the following assertion.

\begin{tw}\label{t:nota}
Let $(r_i)_{i\in \mathbb N}$ be graphs with all vertex degrees at least 3.
Let $G$ be given by graphical presentation with relators $(r_i)_{i\in \mathbb N}$ 
satisfying the lacunary walling condition. Then $G$ has the Haagerup property but is not coarsely amenable. In particular,
$G$ admits a coarse embedding into a Hilbert space but $G$ is not coarsely amenable. 
\end{tw}

The lacunary walling condition (see Definition~\ref{d:lwpres}) assures some upper bound on all vertex degrees and girth tending to infinity as $i\to\infty$.
Coarse non-amenability of such a group $G$ follows then by a result of Willet~\cite{Willett} 
combined with Lemma~\ref{l:remb},
and the fact that coarse amenability is inherited by subspaces, see e.g.~\cite{NowakYu}*{Proposition 4.2.5}. 
The Haagerup property of $G$ follows by our Main Theorem and the coarse embedding into a Hilbert space is provided by an orbit map of the 
given proper affine isometric action.

The question above and the interplay of coarse amenability versus coarse embeddings and the Haagerup property are still
challenging  as an explicit construction of such a sequence  $(r_i)_{i\in \mathbb N}$ is yet an open problem. 
Specifically, the existence of a graphical small cancellation labeling (see the next Section for the terminology) is of great interest.

\begin{open}
Is there a sequence $(r_i)_{i\in \mathbb N}$ of graphs with all vertex degrees between 3 and some upper bound
such that
some graphical presentation with relators $(r_i)_{i\in \mathbb N}$  satisfies the $C'(1/6)$--small cancellation condition? 
\end{open}

A natural candidate for a sequence $(r_i)_{i\in \mathbb N}$  as in Theorem~\ref{t:nota} is
Arzhantseva-Guentner-\v{S}pakula's box space~\cite{AGS}. This space is the first example of a metric space with bounded geometry 
(indeed, of a regular graph) which coarsely embeds into a Hilbert space but which is not coarsely amenable. 
In~\cite{AGS}, a wall structure on this graph is provided. We show in Example 3 that this wall structure indeed satisfies 
our $\beta$--condition  and, after a slight variation of the construction, our
$\delta$--condition (Definition~\ref{d:sep}), both required by the lacunary walling condition.

In quest of a sequence of graphs $(r_i)_{i\in \mathbb N}$ as above, we prove the following general result, of independent interest,
as it provides the first -- explicit -- small cancellation labeling of an appropriate subdivision of any given sequence of bounded degree graphs. 

\begin{tw}\label{t:sc}
Let $\ov \Gamma = (\Gamma_i)_{i\in \mathbb N}$ be a family of finite graphs with degree bounded by $d$.
For every $\ov \lambda = (\lambda_i)_{i\in \mathbb N}$ with $\lambda_i\in (0,1)$, there exists a sequence $(j_i)_{i\in \mathbb N}$ of natural numbers with the following property.
There exists an explicit labeling of the family $(\Gamma_i^{j_i})_{i\in \mathbb N}$ of subdivisions by $d+2$ letters satisfying the $C'(\ov \lambda)$--small cancellation condition.
\end{tw}

Moreover, in Theorem~\ref{p:scsubl}, for a given infinite family of graphs we describe a method of labeling it (modulo a subdivision) 
in such a way that a stronger \emph{lacunarity} condition (see Definition~\ref{d:cond}) is satisfied.
This is then used to provide examples of non-classical infinite graphical small cancellation presentations of groups with the Haagerup property.

\medskip 

\noindent
{\bf Organisation.} 
In \S~\ref{s:prelim}, we define small cancellation complexes and graphical presentations.
In \S~\ref{s:wall}, we describe the structure of a space with walls on small cancellation complexes and introduce
the ($\beta,\delta$)--separation condition, required for the definition of the lacunary walling condition.
In \S~\ref{s:properness}, we define the lacunary walling condition and prove that a complex with the lacunary walling condition satisfies the linear separation property, Theorem~\ref{t:linear}. In \S~\ref{s:Haa}, we deduce Main Theorem. 
In \S~\ref{s:app}, we prove Theorem~\ref{t:sc}.
Then we give concrete examples of graphical small cancellation presentations that satisfy the hypothesis of our Main Theorem
and that do not reduce to classical small cancellation presentations. This is done in \S~\ref{s:ex}, where
we also discuss the box space from~\cite{AGS}. Finally, we show that our lacunary walling condition, crucial for the proof of Theorem~\ref{t:linear}, cannot be removed, see~\S~\ref{optim}.\smallskip

\medskip

\noindent
{\bf Acknowledgments.}  We thank Rufus Willett for comments on the Baum-Connes conjectures and Jan \v{S}pakula for the discussion on the box spaces in Example 3.

\section{Small cancellation complexes}
\label{s:prelim}

\subsection{\texorpdfstring{$C'(\lambda)$}{C'(lambda)}--complexes.}\label{s:c'cpl}
Here we describe the spaces that we  work on throughout this paper. 
Let $X^{(1)}$ be a graph. Our graphs have no loops neither multiple edges, and moreover, all our graphs have uniformly bounded degree.
Let $(\varphi_i \colon r_i \to \xj)_{i\in \mathbb N}$ be a family of local isometries of finite graphs $r_i$.
We call these finite graphs \emph{relators}.
We assume that $\varphi_i \neq \varphi_j$, for $i\neq j$.
The \emph{cone} over the relator $r_i$ is the quotient space $\mr{cone}\, r_i:=(r_i \times [0,1])/\{ (x,1) \sim (y,1)\}$.
The main object of our study is the \emph{coned-off space}:
\begin{align*}
X:=X^{(1)}\cup_{(\varphi_i)} \bigcup_{i\in \mathbb N} \mr{cone}\,r_i,  
\end{align*}
where $\varphi_i$ is the map $r_i\times \{ 0 \} \to \xj$.
We assume that $X$ is simply connected.
The space $X$ has a natural structure of a CW complex, or even of a simplicial complex, however we will not specify it.
Nevertheless, we usually call $X$ a ``complex". Throughout the article, if not specified otherwise, we consider the \emph{path metric}, denoted by $d(\cdot,\cdot)$, defined on the $0$--skeleton $X^{(0)}$ of $X$ by (combinatorial) paths in the $1$--skeleton $X^{(1)}$. \emph{Geodesics} are the shortest paths in $X^{(1)}$ for this metric.

\medskip
A path $p \hookrightarrow X^{(1)}$ is a \emph{piece} if there are relators $r_i,r_j$ such that $p\hookrightarrow X$ factors as $p \hookrightarrow r_i \stackrel{\varphi_i}{\longrightarrow} X$ and as $p\hookrightarrow r_j \stackrel{\varphi_j}{\longrightarrow}  X$, but there is no isomorphism $r_i \to r_j$ that makes the following diagram commutative.
$$  \begindc{\commdiag}[2]
\obj(12,1)[a]{$r_i$}
\obj(35,2)[b]{$X$}
\obj(35,1)[b']{}
\obj(35,16)[c]{$r_j$}
\obj(35,17)[c']{}
\obj(12,16)[d]{$p$}
\obj(12,17)[d']{}
\mor{a}{b'}{}
\mor{c}{b}{}
\mor{a}{c}{}
\mor{d'}{c'}{}
\mor{d'}{a}{}
\enddc  $$

\noindent
 This means that $p$ occurs in $r_i$ and $r_j$ in two essentially distinct ways.

Let $\lambda \in (0,1)$.
We say that the complex $X$ satisfies the \emph{$C'(\lambda)$--small cancellation condition} (or, shortly, the \emph{$C'(\lambda)$--condition}; or we say that $X$ is a \emph{$C'(\lambda)$--complex}) if every piece $p\hookrightarrow X$ factorizing through $p\hookrightarrow r_i \stackrel{\varphi_i}{\longrightarrow} X$ has length strictly less than $\lambda \cdot \gi\,r_i$. 

In this paper, we use the following stronger variant of the 
$C'(\lambda)$--small cancellation condition. Let $\overline {\lambda}:=(\lambda (r_i))_{i\in \mathbb N}$ be a vector with $\lambda (r_i) \in (0,1)$.
We say that the complex $X$ satisfies the \emph{$C'(\overline{\lambda})$--small cancellation condition} if 
every piece $p\hookrightarrow X$ factorizing through $p\hookrightarrow r_i \stackrel{\varphi_i}{\longrightarrow} X$ has length strictly less than $\lambda (r_i) \cdot \gi\,r_i$.

\begin{lem}[Relators embed]
\label{l:remb}
If $X$ is a $C'(1/24)$--complex, then the maps $\varphi_i\colon r_i \to X$ are isometric embeddings. 
\end{lem}
\begin{proof}
This follows from results in~\cite{W-qch} (see also \cite{W-ln}*{Chapter 10} or \cite{Oll}*{Theorem 1}). Indeed, by the proof of Lemma 3.46 there, $X$ has \emph{short innerpaths}, and by Theorem 3.31 the nonpositive curvature Condition 5.1.(2) is satisfied (see also \cite{W-ln}*{Lemma 9.12}). Thus, the claim follows from Lemma 5.5.
\end{proof}

In our case, of graphical small cancellation, the preceding result can be strengthened to $C'(1/6)$--complexes. We do not elaborate on this
since anyway we will use $C'(\lambda)$--complexes as $\lambda \to 0$.

\subsection{\texorpdfstring{$C'(\lambda)$}{C'(lambda)}--presentations}\label{s:c_prime_gp}
Let $Y^{(1)}$ be a finite graph and let $(\varphi_i \colon r_i \to Y^{(1)})_{i\in \mathbb N}$ be a family of local isometries of graphs. They form a
\emph{graphical presentation}
\begin{align}\label{eq:gpres} 
\langle Y^{(1)}\; | \; (r_i)_{i\in \mathbb N} \rangle,
\end{align}
defining a group $G:=\pi_1(Y^{(1)})/\langle\langle \pi_1(r_i)_{i\in \mathbb N} \rangle\rangle$, being the fundamental group of a coned-off space:
\begin{align*}
Y:=Y^{(1)}\cup_{(\phi_i)} \bigcup_{i\in \mathbb N} \mr{cone}\,r_i.  
\end{align*}
We say that $\langle Y^{(1)}\; | \; (r_i)_{i\in \mathbb N} \rangle$ is a \emph{$C'(\lambda)$--small cancellation presentation} (respectively, 
a \emph{$C'(\overline{\lambda})$--small cancellation presentation}) if the universal cover $X$ of $Y$, with the induced maps $(\varphi_i \colon r_i \to X^{(1)})_{i\in \mathbb N}$, is so.

\subsection{Local-to-global density principle}\label{s:prel}

Here we provide a simple tool that allows us to conclude global properties of complexes from the local ones. Its proof can be found in~\cite{AO}.

Let $\gamma$ be a simple path in $X^{(1)}$.
For a subcomplex $B$ of $\gamma$, by $E(B)$ we denote the set of edges of $B$.
Let $\mathcal U$ be a family of nontrivial subpaths of $\gamma$, and let $A$ be a subcomplex of $\bigcup \mathcal U$ (that is, of the union $\bigcup_{U\in \mathcal U}U$).

\begin{lem}[Local-to-global density principle]
\label{l:lgd}
Assume that there exists $C\geqslant 0$, such that
\begin{align*}
  \frac{|E(A)\cap E(U)|}{|E(U)|}\geqslant C,
\end{align*}
for every $U\in \mathcal U$.
Then $|E(A)|\geqslant (C/2)|E(\bigcup \mathcal U)|$.
\end{lem}

\section{Walls}\label{s:wall}

Let $X$ be a $C'(\lambda)$--complex.
In this section, we equip the $0$--skeleton $X^{(0)}$ of $X$ with the structure of a space with walls $(X^{(0)}, \mathcal{W})$. We use a method of constructing walls from~\cite{W-qch}.
\medskip

Recall, cf.\ e.g.\ \cite{ChMV}, that for a set $Y$ and a family $\mathcal W$ of partitions (called \emph{walls}) of $Y$ into two parts, the pair $(Y,\mathcal W)$ is called a \emph{space with walls} if the following holds. For every two distinct points $x,y\in Y$ the number of walls separating $x$ from $y$ (called the \emph{wall pseudo-metric}), denoted by $\dw(x,y)$, is finite.
\medskip

Now we define walls for $\xz$. For a tentative abuse of notation we denote by ``walls" some subsets of edges of $\xj$, then showing that they indeed define walls.
Roughly speaking, two edges are in the same wall if they are ``opposite" in a relator, and then this relation is transitively closed.
This general idea goes back to the definition of walls for the classical small cancellation theory; see \cites{W-sc,AO} for finite and infinite classical small cancellation complexes, respectively.
However, since there is no notion of ``oppositeness" in a general graph relator $r_i$, we require certain assumptions on each graph $r_i$.

\subsection{Defining walls}
\label{s:relwall}

A \emph{wall} in a graph $\Gamma$ is a collection $w$ of edges such that removing all open edges of $w$ decomposes $\Gamma$ in
exactly two connected components. In particular, it requires $\Gamma$ to be connected. We call $\Gamma$ a \emph{graph with walls}, if every edge belongs to a wall. 

If not stated otherwise, we assume that for a $C'(1/24)$--complex $X$, with given relators $r_i$, each graph $r_i$ is  a graph with walls. 
Observe that every $r_i$ is in fact an isometrically embedded subgraph of $X$, by Lemma~\ref{l:remb}.
Following \cite[Section 5]{W-qch} (see also \cite[Chapter 10]{W-ln}), we define walls in $\xj$ as follows: Two edges are in the same wall if they are in the same wall in some relator $r_i$. This relation is then 
extended transitively for all relators.

In general, the above definition may not result in walls for $\xj$ --- they might not decompose $\xj$, etc. We require some further assumptions on
walls in relators, which are formulated in the next section. Then, in Section~\ref{s:walls}, we prove that our definition of walls in $\xj$ makes sense, and
we explore further properties of such a system of walls.

\subsection{Separation property}
\label{s:sep}

\begin{de}[($\beta,\delta$)--separation]
\label{d:sep}
For $\beta \in (0,1/2]$ and $\delta \in (0,1)$ a graph $r$ with walls satisfies the \emph{($\beta,\delta$)--separation property}
if the following two conditions hold:
\medskip

\noindent
\emph{\underline{$\beta$--condition}}: for every two edges $e,e'$ in $r$ belonging to the same wall we have
\begin{align*}
d(e,e')+1\geqslant \beta \cdot \gi\,r.
\end{align*}

\noindent
\emph{\underline{$\delta$--condition}}: for every geodesic $\gamma$ in $r$, the number of edges in $\gamma$ whose walls meet $\gamma$ at 
least twice is at most $\delta\cdot |\gamma|$.
\medskip

A complex $X$ satisfies the \emph{($\beta,\delta$)--separation property} if every its relator does so.
\end{de}

There are other ways of defining an analogue of the $\beta$--condition above, which would make the $\delta$--condition 
unnecessary. However, one requires then large $\beta$, which is not convenient for providing examples.


\subsection{Walls in $X$.}
\label{s:walls}

Let us show that a $C'(\lambda)$--complex $X$ satisfying the \bds  property, 
does possess the wall structure given by the walls as defined in Subsection~\ref{s:relwall}, for sufficiently small $\lambda\leqslant 1/24$.
We use results of Wise \cite[Section 5]{W-qch} (see also \cite[Chapter 10]{W-ln}). In particular, we have to check that $X$ satisfies the assumptions from~\cite{W-qch}. 

\begin{lem}[Generalized $B(6)$]
\label{l:b6}
Let $X$ be a complex satisfying the $\beta$--condition from Definition~\ref{d:sep}.
Then there exists $\lambda\leqslant 1/24$ with the following property.
If $X$ satisfies the $C'(\lambda)$--condition then 
$X$ satisfies the \emph{generalized $B(6)$ condition} of \cite[Definition 5.1]{W-qch}.
\end{lem}
\begin{proof}
Condition (1) of \cite[Definition 5.1]{W-qch} follows immediately from our definition of the coned-off space $X$. Condition
(2) follows from the $C'(1/24)$--condition (see the proof of Lemma~\ref{l:remb} above).
Conditions (3) and (6) follow from our definition of walls in graph relators.
For a given $\beta$ the conditions (4) and (5) are implied by the $\beta$--condition together with the $C'(\lambda)$--condition, provided $\lambda$ is sufficiently small.
\end{proof}

For the rest of this subsection we assume that $X$ satisfies the $\beta$--condition, for some $\beta\in (0,1)$, and 
the $C'(\lambda)$--small cancellation condition for $\lambda$ as in Lemma~\ref{l:b6}.
With this fact in hand we  use~\cite[Section 5]{W-qch} in our setting.

\begin{lem}[{\cite[Remark 5.24]{W-qch}} and {(\cite[Theorem 10.1]{W-ln})}]
  \label{l:wallsep}
  Removing all open edges from a given wall decomposes $\xj$ into exactly two connected components.
\end{lem}

Thus, we define the family $\mathcal W$ for $\xz$ as the partitions of $\xz$ into sets of vertices in the connected components described by the lemma above.

\begin{prop}
  \label{p:sww}
  With the system of walls defined as above, $(\xz,\mathcal W)$ is a space with walls.
\end{prop}
\begin{proof}
  Since, for any two vertices, there exists a path in $\xj$ connecting them, we get that the number of walls separating those two vertices is finite.
\end{proof}

We recall further results on walls that will be extensively used in Section~\ref{s:properness}.

For a wall $w$, its \emph{hypergraph} $\Gamma_w$ is a graph defined as follows (see \cite[Definition 5.18]{W-qch} and \cite{W-sc}). 
There are two types of vertices in $\Gamma_w$ (see e.g.\ Figure~\ref{f:C}):
\begin{itemize}
\item
\emph{edge-vertices} correspond to edges in $w$,
\item
\emph{relator-vertices} correspond to relators containing edges in $w$.  
\end{itemize}
An \emph{edge} in $\Gamma_w$ connects an edge-vertex with a relator-vertex whenever the corresponding
relator contains the given edge.

\begin{lem}[{\cite[Theorem 5.19]{W-qch}}]
\label{l:hypergraph}
Each hypergraph is a tree.
\end{lem}

The \emph{hypercarrier} of a wall $w$ is the $1$--skeleton of the subcomplex of $X$ consisting of all relators containing edges in $w$ or of a single edge $e$ if $w=\{ e \}$.

\begin{tw}[{\cite[Corollary 5.34]{W-qch}}]
  \label{l:carrconv}
  Each hypercarrier is a convex subcomplex of $\xj$, that is, any geodesic connecting vertices of a hypercarrier is contained in this hypercarrier.
\end{tw}

The following result is implicit in \cite[Section 5]{W-qch}, and formally it follows from Lemmas~\ref{l:hypergraph} \& \ref{l:carrconv} above, and from
\cite[Lemma 5.9]{W-qch}.

\begin{cor}
\label{l:relcon}
Relators are convex subcomplexes of $\xj$.
\end{cor}

\section{Linear separation property}
\label{s:properness}
\noindent
In this section, we perform the main step toward Main Theorem, namely, we prove Theorem~\ref{p:lsp} below (Theorem~\ref{t:linear} from Introduction). This implies the properness of the wall pseudo-metric and it is the most involved part of the paper (cf.\ also Remark after Definition~\ref{d:cond} and Section~\ref{optim} below). 
\medskip

\emph{From now on, unless stated otherwise, each complex $X$ considered in this paper has relators $(r_i)_{i\in \mathbb N}$ being graphs with walls, and satisfies the following lacunary walling condition}.

\begin{de}[{Lacunary walling}]
\label{d:cond}
Let $\beta \in (0,1/2], \delta \in (0,1), M\in (0,1), K >0$,  and let $k$ be a natural number larger than $1$. 
Let $\overline{\lambda}=(\lambda (r_i))_{i\in \mathbb N}$ be a vector with $\lambda (r_i)\leqslant \lambda$, where $\lambda< \beta/2$ is the constant
from Lemma~\ref{l:b6} (that is, such that $X$ satisfies the generalized $B(6)$ condition). 
We denote by $b_{r_i}(t)$ the maximal number of edges in a ball of radius
$t$ in the graph $r_i$. We say that $X$ satisfies the \emph{lacunary walling condition} if:
 
\begin{itemize}
\item
$X^{(1)}$ has degree bounded by $k$;
\item
(Small cancellation condition) $X$ satisfies the $C'(\overline{\lambda})$--condition;
\item
(Lacunarity) $b_{r_i}(\lambda (r_i) \cdot \gi\,r_i) \leqslant K \cdot \gi\,r_i$;
\item
(Separation condition) $X$ satisfies the \bds property;
\item
(Compatibility) $(1-\delta)(\beta - \lambda(r_i)) - 2K - 4\lambda(r_i) \geqslant M \cdot (\beta - \lambda(r_i))$.
\end{itemize}  
\end{de}

Observe that this definition makes sense, that is, there are choices of all the constants and functions above
satisfying the given constraints. 
To see this, note that (in the compatibility condition):
\begin{align*}
(1-\delta)(\beta - \lambda(r_i)) - 2K-4\lambda(r_i) = \left(1 - \delta-\frac{2K+4\lambda(r_i)}{\beta - \lambda(r_i)}\right)(\beta - \lambda(r_i)).
\end{align*}
Thus, after setting $\beta$ and $\delta$, one can choose small $K$ and $\ov \lambda$ so that the compatibility condition holds.
Then one may further decrease the function $\ov \lambda$ to satisfy the lacunarity condition.
\medskip

\rems
1) Our assumptions are not quantitatively optimal, they suit our general goal toward explicit examples (cf.\ Section~\ref{s:ex}).
However, in Section~\ref{optim}, we argue that the lacunary walling condition is in a sense necessary in our approach.
\smallskip

\noindent
2) In this paper, we follow (up to some notations) the construction of walls provided by Wise \cite{W-qch} in a much
more general case of small cancellation over CAT(0) cubical complexes. In fact, for graphical small cancellation --- as considered in our
paper --- one could adapt the proofs provided in \cite{W-sc} in the classical small cancellation case (cf.\ e.g.\ \cite{OllivierWise}).
We decided to follow the more general approach having in mind possible future extensions of our results.
\smallskip

\noindent
3) Whereas the construction of walls for (cubical) small cancellation complexes is entirely the idea of Wise, the 
properness of the wall pseudo-metric is proved only in some cases in \cites{W-sc,W-qch}. In particular, as we point out in 
\cite[Section 6]{AO}, there exist classical $B(6)$ small cancellation complexes whose wall pseudo-metric is not proper.
In \cite[Section 5.k]{W-qch} the linear separation property is proved in the case of graphical small cancellation under the additional
assumption on the presentation finiteness. The proof does not extend to our -- infinitely presented -- case.
In Section~\ref{optim}, we comment on relations between our approach and the one from  \cite[Section 5.k]{W-qch}.
 
\medskip

By Subsection~\ref{s:walls}, for a complex $X$ with the lacunary walling condition, there is a structure of space with walls $(X^{(0)},\mathcal W)$.
The rest of this section is devoted to proving that $(X^{(0)},\mathcal W)$ satisfies the \emph{linear separation property} (Theorem~\ref{t:linear} in the Introduction, and Theorem~\ref{p:lsp} below) stating that the wall pseudo-metric on $\xz$ is bi-Lipschitz equivalent to the path metric (cf.\ e.g.\ \cite[Section 5.11]{W-qch}). 
\medskip

Let $p,q$ be two distinct vertices in $X$.
It is clear that
\begin{align*}
  \dw(p,q) \leqslant d(p,q).
\end{align*}
For the rest of this section our aim is to prove an opposite (up to a scaling constant) inequality.
\medskip

Let $\gamma$ be a geodesic in $X$ (that is, in its $1$--skeleton $\xj$) with endpoints $p,q$. Let $A(\gamma)$ denote the set of edges in $\gamma$ whose walls meet $\gamma$ in only one edge (in particular such walls separate $p$ from $q$). Clearly $\dw(p,q)\geqslant |\ag|$.
We thus estimate $\dw(p,q)$ by closely studying the set $\ag$. The estimate is first provided locally (in Subsection~\ref{s:local} below) and then we use the local-to-global density principle (Lemma~\ref{l:lgd}) to obtain a global bound.
\medskip

We begin with an auxiliary lemma.
Let $r$ be a relator. Since $r$ is convex in $X$, its intersection with $\gamma$ is an interval $p'q'$, with $p'$ lying closer
to $p$ --- see Figure~\ref{f:C}. 
Consider the set $C$ of edges $e$ in $p'q'$, whose walls $w$ meet $\gamma$ at least twice and, moreover, have the following properties. Let $e'\in w$ (considered as an edge-vertex in the hypergraph $\Gamma_w$) be a closest vertex to $e$ in $\Gamma_w$, among edges of $w$ lying on $\gamma$. In the hypergraph $\Gamma_w$ of the wall $w$, which is a tree by Lemma~\ref{l:hypergraph}, consider the unique geodesic $\gamma_w$ between vertices $e$ and $e'$. We assume that there are at least two distinct relator-vertices on $\gamma_w$, one of them being~$r$.

\begin{figure}[h!]
\centering
\includegraphics[width=0.9\textwidth]{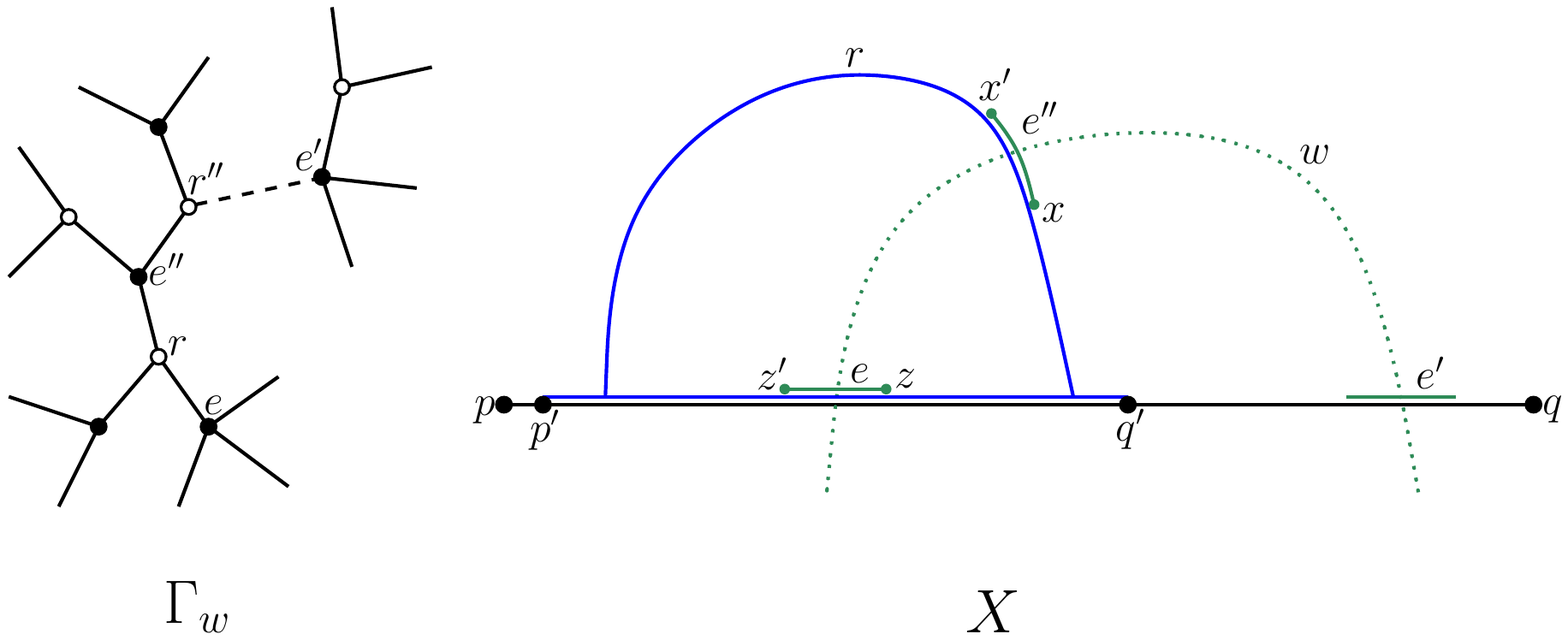}
\caption{Lemma~\ref{l:C}.}
\label{f:C}
\end{figure}

\begin{lem}
\label{l:C}
In the situation as above we have $|C|\leqslant 2\cdot b_r(\lambda(r)\cdot \mr{girth}\,r)$.
\end{lem}
\begin{proof}
Suppose that $q'$ lies between $e$ and $e'$ (on $\gamma$). 
Let $e''\neq e$ be the edge-vertex on $\gamma_w$ adjacent to $r$ and, consequently, let $r''$ be the relator-vertex on $\gamma_w$ 
adjacent to $e''$ --- see Figure~\ref{f:C}.
By convexity (Lemma~\ref{l:carrconv}) and the tree-like structure (Lemma~\ref{l:hypergraph}) of the hypercarrier of $w$, containing $e$ and $e'$, we have that $q'\in r''$. Since $r\cap r''$ is convex and contains both
$e''$ and $q'$, by the small cancellation condition we have
\begin{align*}
d(e'',q')+1\leqslant \lambda(r)\cdot \mr{girth}\, r.
\end{align*}
Therefore, the number of edges $e''$ as above is at most $b_r(\lambda(r)\cdot \mr{girth}\,r)$.
The same number bounds the quantity of the corresponding walls.
By our assumptions, every such wall contains only one edge in $p'q'$.
Thus, the number of edges $e$ as above is at most $b_r(\lambda(r)\cdot \mr{girth}\,r)$.
Taking into account the situation when $p'$ lies between $e$ and $e'$ we have
\begin{align*}
|C|\leqslant 2\cdot b_r(\lambda(r)\cdot \mr{girth}\,r).
\end{align*} \end{proof}


\subsection{Local estimate on \texorpdfstring{$|\ag|$}{|A(gamma)|}.}
\label{s:local}
For a local estimate we need to define neighborhoods $N_e$ -- \emph{relator neighborhoods in $\gamma$} -- one for every edge $e$ in $\gamma$, for which the number $|E(N_e)\cap \ag|$ can be bounded from below.

For a given edge $e$ of $\gamma$ we define a corresponding relator neighborhood $N_e$ as follows.
If $e\in \ag$ then $N_e=\{ e \}$. Otherwise, we proceed in the way described below.
\medskip

Since $e$ is not in $\ag$, its wall $w$ crosses $\gamma$ in at least one more edge. In the wall $w$, choose an edge $e'\subseteq \gamma$ being  a closest edge-vertex to $e\neq e'$ in the hypergraph $\Gamma_w$ of the wall $w$. 
We consider separately the two following cases, see Subsection~\ref{s:c1} and Subsection~\ref{s:c2} below.
\subsubsection{(\emph{Case I.)} The edges $e$ and $e'$ do not lie in common relator.}
\label{s:c1}
In the hypergraph $\Gamma_w$ of the wall $w$, which is a tree by Lemma~\ref{l:hypergraph}, consider the geodesic $\gamma_w$ between vertices $e$ and $e'$. Let $r$ be the relator-vertex in $\gamma_w$ adjacent to $e$. Let $e''$ be an edge-vertex in $\gamma_w$ adjacent to $r$. Consequently, let $r''$ be the other relator-vertex in $\gamma_w$ adjacent to $e''$.
The intersection of $r$ with $\gamma$ is an interval $p'q'$.
Assume without loss of generality, that $q'$ lies between $e$ and $e'$ ---  see Figure~\ref{f:C}. 
\medskip
 
We define the relator neighborhood $N_e$ as the interval $p'q'=r\cap \gamma$.

\begin{lem}
\label{l:Ne}
\begin{align*}
|E(N_e)|> (\beta - \lambda(r))\cdot \gi\,r.
\end{align*}  
\end{lem}
\begin{proof}
Let $xq'$ be the geodesic between $e''$ and $q'$.  Let $z$ be the vertex in $e$ closest to $q'$ ---  see Figure~\ref{f:C}.
By the \bds property we have
\begin{align}
\label{e:110}
|xq'|+|q'z|+1\geqslant \beta \cdot \gi\,r.
\end{align}  
On the other hand, by the small cancellation condition, we have
\begin{align}
\label{e:115}
|xq'|+1\leqslant \lambda(r)\cdot \gi\,r.
\end{align}  
Combining (\ref{e:110}) with (\ref{e:115}), we obtain 
\begin{align*}
|p'q'|&\geqslant |q'z|+1\geqslant \beta \cdot \gi\, r -|xq'| > \beta \cdot \gi\, r - \lambda(r)\cdot \gi\,r \\
& \geqslant (\beta - \lambda(r))\cdot \gi\,r.
\end{align*}  
\end{proof}

We are now ready to state the main result in Case I.

\begin{lem}[Local density of $\ag$ --- Case I]
  \label{l:prop2}
The number of edges in $N_e$, whose walls separate $p$ from $q$ is estimated as follows:
\begin{align*}
|E(N_e)\cap \ag| \geqslant \frac{(1-\delta)\cdot (\beta - \lambda (r))-2K -4\lambda (r)}{\beta - \lambda (r)}|E(N_e)|.
\end{align*}
\end{lem}
\begin{proof}
  To estimate $|E(N_e)\cap \ag|$, that is,
  the number of edges in $N_e$ that belong to $\ag$, we explore the set of edges $f$ in $N_e$ 
outside $\ag$.
We consider separately the three ways in which an edge $f$  of $N_e$ may fail to belong to $\ag$ --- these are studied in Cases: B, C and D below.

Since $f\notin \ag$ there exists another edge of the same wall $w_f$ in $\gamma$. Let $f'$ be a closest to $f$ such edge-vertex in
the hypergraph $\Gamma_{w_f}$. Denote by $\gamma_{w_f}$ the geodesic in $\Gamma_{w_f}$ between $f$ and $f'$.
Let $r_f$ be the relator-vertex on $\gamma_{w_f}$ adjacent to $f$.
\medskip

\noindent
\emph{Case B: There is only one relator-vertex  between $f$ and $f'$ on $\gamma_{w_f}$, and $r_f=r$.}
By convexity of relators, the segment $p'q'$ is geodesic in $r$. Thus, by the  \bds property, 
the cardinality of the set $B$ of such edges $f$ is bounded by
\begin{align}
\label{e:200}
|B|\leqslant \delta \cdot |E(N_e)|.
\end{align}  

\begin{figure}[h!]
\centering
\includegraphics[width=0.9\textwidth]{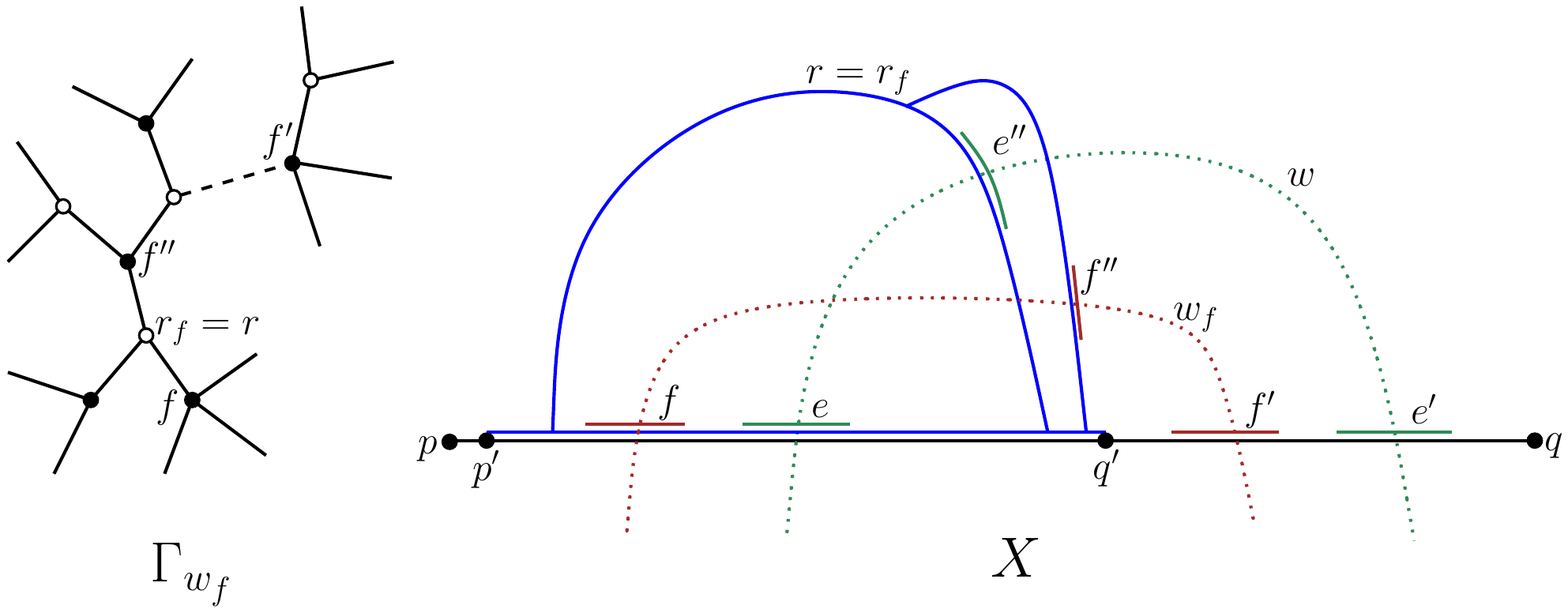}
\caption{Lemma~\ref{l:prop2}, Case I(C).}
\label{f:CC}
\end{figure}

\medskip

\noindent
\emph{Case C: There are at least two distinct relator-vertices between $f$ and $f'$ on $\gamma_{w_f}$, and $r_f=r$} ---  see Figure~\ref{f:CC}.
The cardinality of the set $C$ of such edges $f$ is bounded, by Lemma~\ref{l:C}, as follows:
\begin{align}
\label{e:205}
|C|\leqslant 2\cdot b_r(\lambda(r)\cdot \mr{girth}\,r).
\end{align}

\medskip

\noindent
\emph{Case D: $r_f\neq r$.}
Let the set of such edges $f$ be denoted by $D$. 
Let $p''q'':=r_f \cap \gamma$, with $p''$ closer to $p$.
\medskip

\begin{figure}[h!]
\centering
\includegraphics[width=0.9\textwidth]{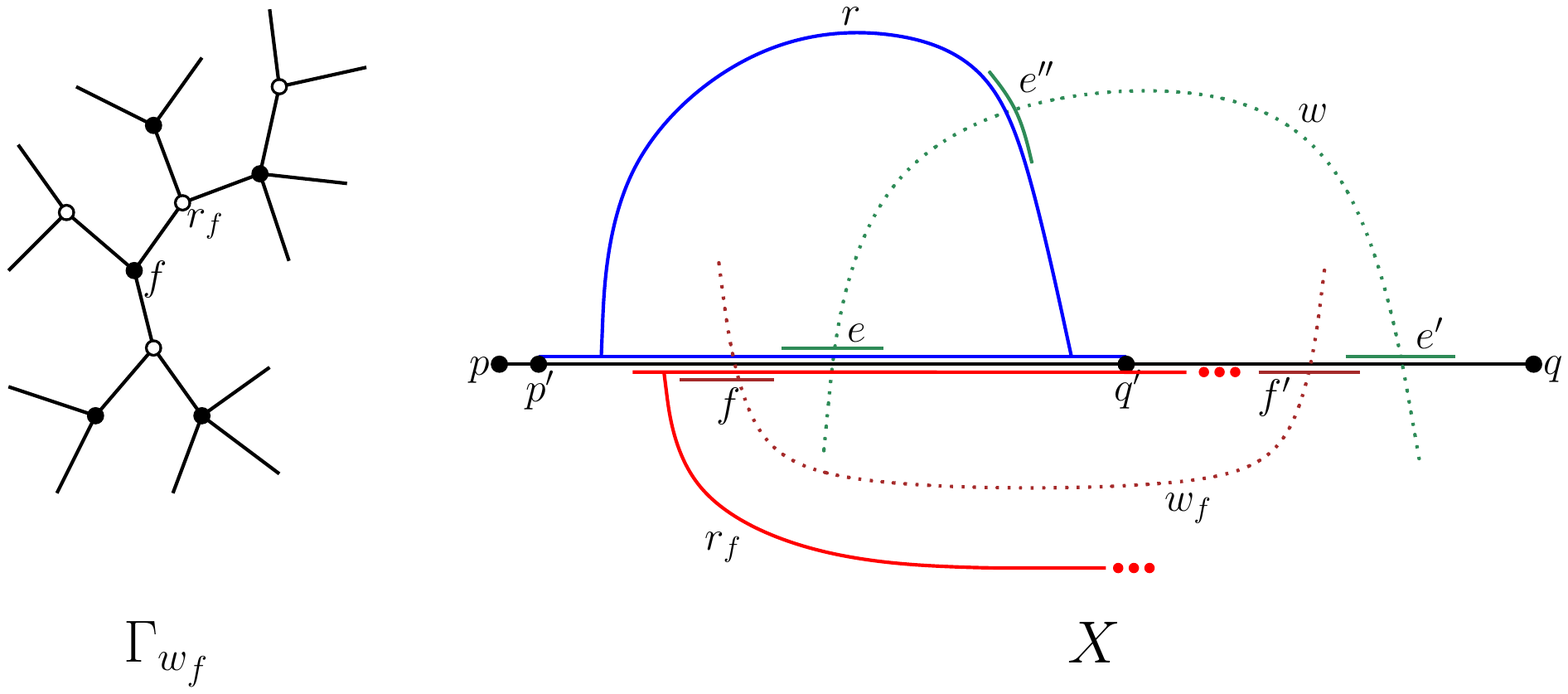}
\caption{Lemma~\ref{l:prop2}, the possible Case I(D).}
\label{f:Dpos}
\end{figure}

First consider the subcase when the relator-vertex $r_f$ is adjacent to both $f$ and $f'$.
Observe that, since $r_f\neq r$, the edge $f'$ does not belong to  $r$.
Without loss of generality, we may assume that $q'$ lies (on $\gamma$) between
$f$ and $f'$. Since $r_f$ is convex, it follows that the interval on $\gamma$ between $f$ and $q'$
is contained in $r_f$ --- see Figure~\ref{f:Dpos}. Thus, by the small cancellation condition, the number of such edges $f$ is bounded
by (taking into account the symmetric situation after exchanging $p'$ and $q'$):
\begin{align}
\label{e:270}
2\cdot \lambda (r)\cdot \gi\,r.
\end{align}  
\medskip

\begin{figure}[h!]
\centering
\includegraphics[width=0.9\textwidth]{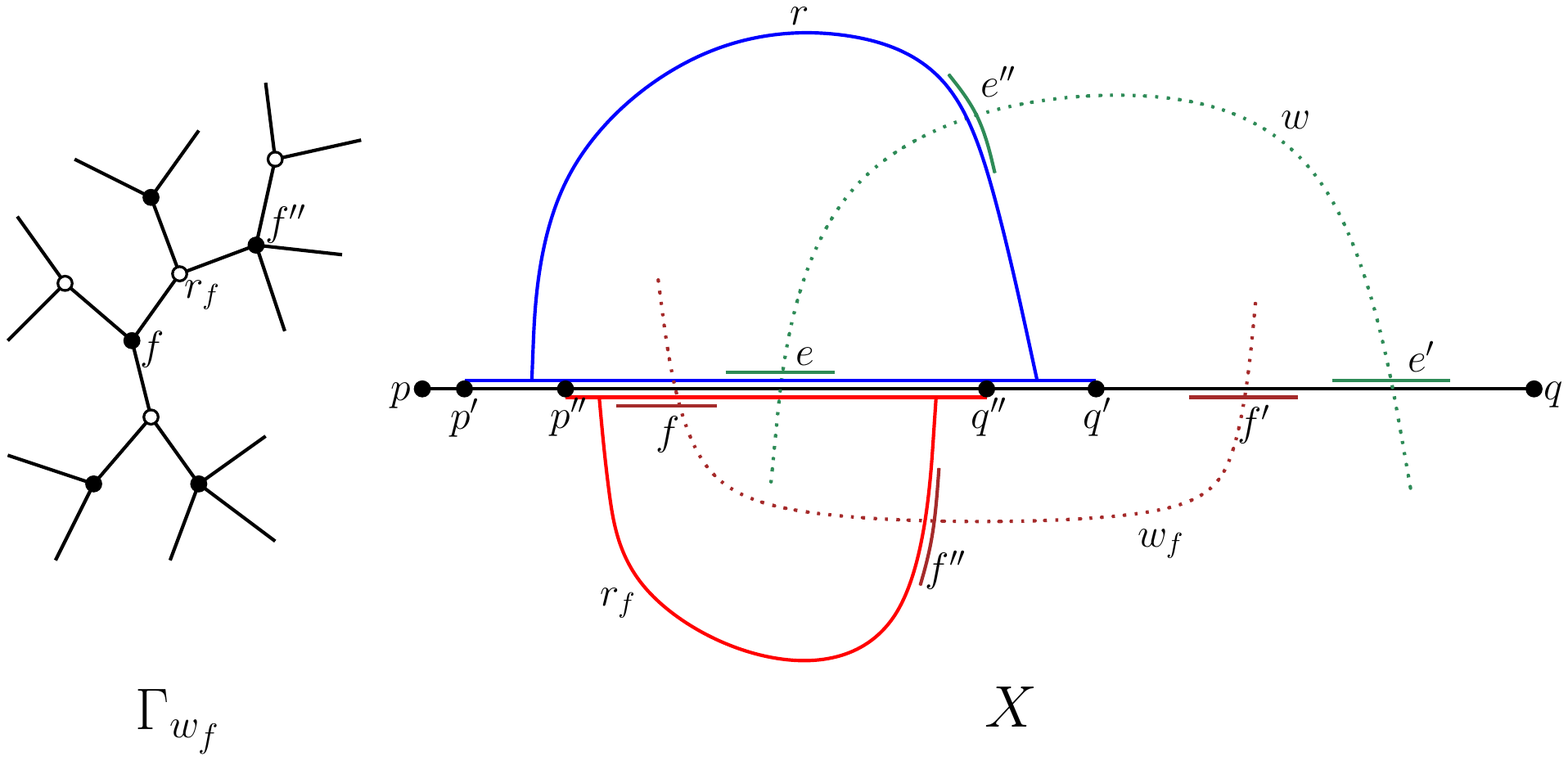}
\caption{Lemma~\ref{l:prop2}, the impossible Case I(D).}
\label{f:Dimpos}
\end{figure}

The other subcase to consider is when there is another relator-vertex $r_f'\neq r_f$ on $\gamma_{w_f}$ adjacent to 
$f''$, that is itself adjacent to $r_f$.
The number of edges $f$ for which $r_f$ contains $p'$ or $q'$ is, again by the small cancellation condition --- see Figure~\ref{f:Dpos}, bounded by:
\begin{align}
\label{e:275}
2\cdot \lambda (r)\cdot \gi\,r.
\end{align}  
Thus, further we assume that $p''q'' \subseteq p'q'$ --- see Figure~\ref{f:Dimpos}. We will show that this is impossible.
By the small cancellation condition, we have 
\begin{align}
\label{e:280}
|p''q''|\leqslant \lambda(r) \cdot \gi\,r.
\end{align}  
On the other hand, by Lemma~\ref{l:Ne} we obtain
\begin{align}
\label{e:285}
|p''q''|\geqslant (\beta -\lambda (r))\cdot \gi\,r.
\end{align}  
Combining (\ref{e:280}) and (\ref{e:285}), we get 
\begin{align*}
\lambda (r) \geqslant \beta - \lambda (r).
\end{align*}   
This however contradicts our choice of $\beta$ and $\ov {\lambda}$ (see Definition~\ref{d:cond}).
\medskip

Combining quantities (\ref{e:270}) and (\ref{e:275}) above we obtain the following bound on the number of
edges in $D$:
\begin{align}
\label{e:290}
|D|\leqslant 4 \cdot \lambda (r)\cdot \gi\,r.
\end{align}  
\medskip

Now we combine all the cases: B, C, and D, to obtain the following bound in Case I, see estimates (\ref{e:200}), (\ref{e:205}), and (\ref{e:290}) above.
\begin{align*}
|E(N_e) \setminus A(\gamma)|&\leqslant  |B|+|C|+|D|\\ & \leqslant \delta \cdot |E(N_e)|+ 2\cdot b_r(\lambda(r)\cdot \mr{girth}\,r) + 4 \cdot \lambda (r)\cdot \gi\,r. 
\end{align*}  
By lacunarity (see Definition~\ref{d:cond}), we have $b_r(\lambda(r)\cdot \mr{girth}\,r) \leqslant K\cdot \gi\,r$, and by Lemma~\ref{l:Ne}, we get $\gi\,r < \frac{|E(N_e)|}{\beta - \lambda (r)}$.
Therefore,
\begin{align}
|E(N_e) \setminus A(\gamma)| \leqslant \frac{\delta\cdot (\beta - \lambda (r))+2K +4\lambda (r)}{\beta - \lambda (r)}|E(N_e)|,
\end{align}  
and hence
\begin{align}
|E(N_e)\cap \ag| \geqslant \frac{(1-\delta)\cdot (\beta - \lambda (r))-2K -4\lambda (r)}{\beta - \lambda (r)}|E(N_e)|.
\end{align}  
\end{proof}

\subsubsection{\emph{(Case II.)} The edges $e$ and $e'$ lie in common relator $r$.} 
\label{s:c2}
We may assume (exchanging $e'$ if necessary) that $e'$ is closest to $e$ (in $X$) among edges in $w$ lying in $r \cap \gamma$.
\medskip

The relator neighborhood $N_e$ is now defined as the interval $p'q'=r\cap \gamma$ --- see Figure~\ref{f:c2}.

\begin{figure}[h!]
\centering
\includegraphics[width=0.9\textwidth]{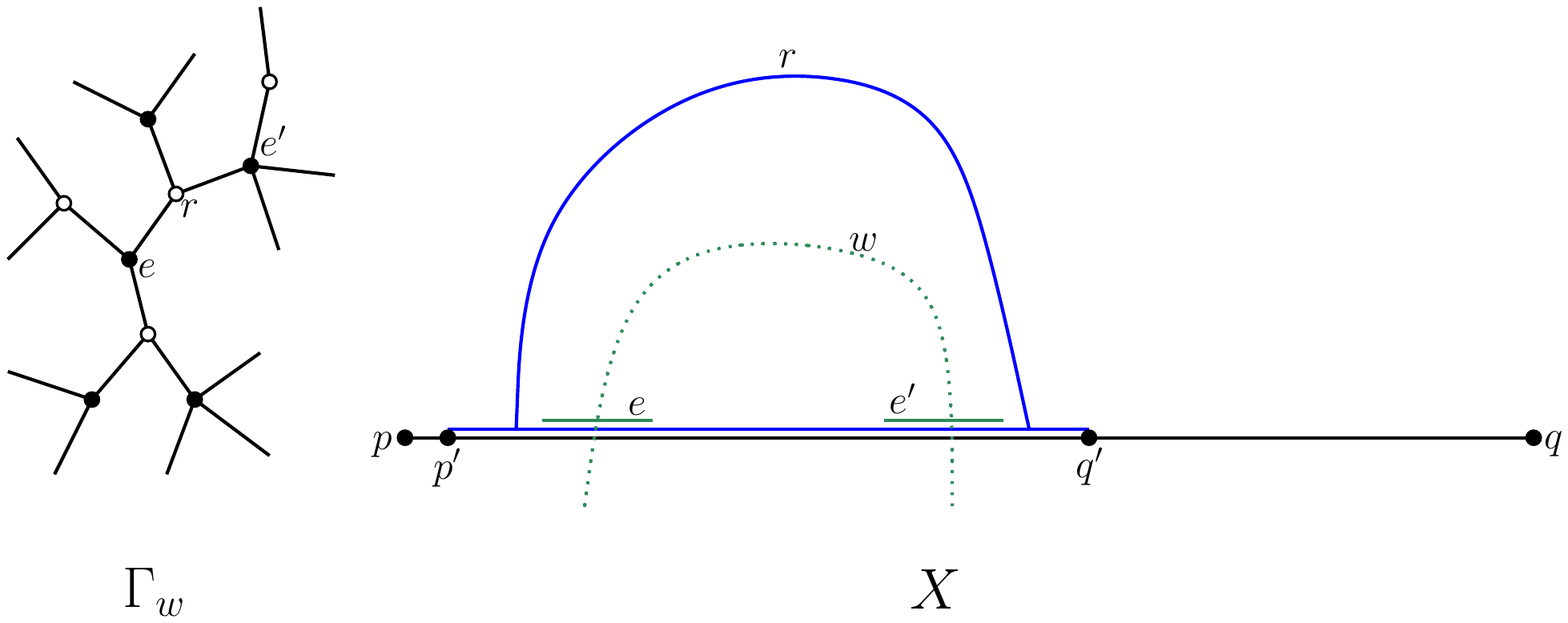}
\caption{Lemma~\ref{l:prop2}, Case II.}
\label{f:c2}
\end{figure}

\begin{lem}[Local density of $\ag$ --- Case II]
  \label{l:prop3}
The number of edges in $N_e$, whose walls separate $p$ from $q$ is estimated as follows:
\begin{align*}
|E(N_e)\cap \ag| \geqslant \frac{(1-\delta)\cdot \beta-2K -4\lambda (r)}{\beta}|E(N_e)|.
\end{align*}
\end{lem}
\begin{proof}
We consider again the set of edges $f$ in $E(N_e)\setminus \ag$. As in Case I (Lemma~\ref{l:prop2}), we consider separately three possibilities: B, C, D
for such an edge $f$ to fail belonging to $\ag$.
The same considerations as in Case I lead to the estimates:
\begin{align*}
|B|&\leqslant \delta \cdot|E(N_e)|,\\
|C|&\leqslant 2\cdot b_r(\lambda(r)\cdot \mr{girth}\,r),\\
|D|&\leqslant 4 \cdot \lambda (r)\cdot \gi\,r.
\end{align*}  
By the \bds property, we have
\begin{align*}
|E(N_e)|\geqslant \beta \cdot \gi\,r. 
\end{align*}   
Combining all the inequalities above we get
\begin{align*}
|E(N_e) \setminus A(\gamma)| \leqslant |B| +|C| +|D|\leqslant \frac{\delta\cdot \beta+2K +4\lambda (r)}{\beta}|E(N_e)|,
\end{align*}  
and hence
\begin{align*}
|E(N_e)\cap \ag| \geqslant \frac{(1-\delta)\cdot \beta-2K-4\lambda (r)}{\beta}|E(N_e)|.
\end{align*}  
\end{proof}

\subsubsection{Final local estimate.}
We are ready to combine all the previous estimates.

\begin{prop}[Local density of $\ag$]
  \label{l:local}
The number of edges in $N_e$, whose walls separate $p$ from $q$ is estimated as follows:
\begin{align*}
|E(N_e)\cap \ag| \geqslant {M}\cdot |E(N_e)|.
\end{align*}
\end{prop}
\begin{proof}
If $e\in \ag$ then the assertion is clear. If $e\notin \ag$ then we use Lemma~\ref{l:prop2} or Lemma~\ref{l:prop3}.
\end{proof}

\subsection{Linear separation property}
\label{s:global}
Using the local estimate on the density of $\ag$ (see Proposition~\ref{l:local}) and the local-to-global density principle (Lemma~\ref{l:lgd}) we now estimate the overall density of edges with walls separating $p$ and $q$, thus obtaining the linear separation property.

\begin{tw}[Linear separation property]
  \label{p:lsp}
  For any two vertices $p,q$ in $X$ we have
  \begin{align*}
    d(p,q)\geqslant \dw(p,q)\geqslant \frac{M}{2}\cdot
    d(p,q),
  \end{align*}
  that is, the path metric and the wall pseudo-metric are bi-Lipschitz equivalent.
\end{tw}

\begin{proof}
  The left inequality is clear. Now we prove the right one.
  Let $\gamma$ be a geodesic joining $p$ and $q$.
  The number $|E(\gamma)|$ of edges in $\gamma$ is equal to $d(p,q)$. On the other hand, the number $|\ag|$ of edges in $\gamma$ whose walls meet $\gamma$ in only one edge is at most $\dw(p,q)$. We will thus bound $|\ag|$ from below.

  For any edge $e$ of $\gamma$, let $N_e$ be its relator neighborhood. The collection $\mathcal U=\{N_e\;|\; e\in E(\gamma)\}$ forms a covering family of subpaths of $\gamma$.
  By the local estimate (Proposition~\ref{l:local}) we have that
  \begin{align*}
    \frac{|\ag\cap E(N_e)|}{|E(N_e)|}\geqslant {M}.
  \end{align*}
  Thus, by the local-to-global density principle (Lemma~\ref{l:lgd}), we have
  \begin{align*}
    |\ag|\geqslant \frac{M}{2}\cdot|E(\gamma)|,
  \end{align*}
  that finishes the proof.
\end{proof}

\section{Main result: the Haagerup property}
\label{s:Haa}


A consequence of the linear separation property (Theorem~\ref{p:lsp}) is the following main result of the paper.

\begin{tw}
  \label{t:haag}
  Let $G$ be a group acting properly on a simply connected $C'(\ov \lambda)$--complex $X$ satisfying the lacunary walling condition. Then $G$ acts properly on a space with 
walls. In particular, $G$ has the Haagerup property.
\end{tw}
\begin{proof}
The group $G$ acts properly on the set of vertices $\xz$ of $X$ equipped with the path metric $d(\cdot,\cdot)$.
By Proposition~\ref{p:sww}, this action gives rise to the action by automorphisms on the space with walls $(\xz,\mathcal W)$.
By the linear separation property (Theorem~\ref{p:lsp}), we conclude that $G$ acts properly on $(\xz,\mathcal W)$.
By an observation of Bo\. zejko-Januszkiewicz-Spatzier \cite{BoJaS} and Haglund-Paulin-Valette (cf.\ \cite{ChMV}), the group $G$ has the Haagerup property.
\end{proof}

\begin{de}[Lacunary walling of a presentation]\label{d:lwpres}
A graphical presentation $$\langle Y^{(1)}\; | \; (r_i)_{i\geqslant 1} \rangle$$
is said to satisfy the \emph{lacunary walling condition} if the universal cover $X$ of the 
coned-off space
$Y^{(1)}\cup_{(\phi_i)} \bigcup_i \mr{cone}\,r_i$  
satisfies the lacunary walling condition introduced in Definition~\ref{d:cond}. 
\end{de}

Observe that Main Theorem follows immediately from the above, since the given group $G$ acts properly on the corresponding universal
cover, as described in Section~\ref{s:prelim}.


\section{An explicit small cancellation labeling of a subdivided family of graphs}
\label{s:app}
The aim of this section is to show that for any (infinite) family of graphs there exists a small cancellation labeling of them,
after subdividing edges in a non-uniform way (Theorem~\ref{p:scsub}). Furthermore, one may enhance the labeling (up to taking a subsequence) to the one satisfying the lacunarity condition (Theorem~\ref{p:scsubl}).
This result allows us (in Section~\ref{s:ex}) to give many -- non-classical -- examples of presentations
satisfying the lacunary walling condition.

All graphs considered in this section are oriented.
For a (oriented) graph $\Gamma$ and $j\in \mathbb N$, by $\Gamma^j$ we denote the \emph{$j$--subdivision} of $\Gamma$, that is
the (oriented) graph obtained by subdividing every edge of $\Gamma$ into $j$ edges, all directed toward the endpoint of the original orientation.

\begin{tw}[Small cancellation labeling of subdivisions]
\label{p:scsub}
Let $\ov \Gamma = (\Gamma_i)_{i\in \mathbb N}$ be a (possibly infinite) family of finite graphs with degree bounded by $d$.
For every $\ov \lambda = (\lambda_i)_{i\in \mathbb N}$ with $\lambda_i\in (0,1)$, there exists a sequence $(j_i)_{i\in \mathbb N}$ of natural numbers with the following property.
There exists a labeling of the family of subdivisions $(\Gamma_i^{j_i})_{i\in \mathbb N}$ by $d+2$ letters satisfying the $C'(\ov \lambda)$--small cancellation condition.
\end{tw}
\begin{proof}
For each $n,k \in \mathbb N$, let $I_{n,k}$ denote the labeling of the segment with all edges directed toward common end-vertex
of length $|I_{n,k}|=kn+k^2/2+k/2$ defined as
(here $a^i$ denotes $a$ labeling $i$ consecutive edges, similarly for $b^j$; orientation from left to right):
\begin{align*}
a^nba^nb^2a^nb^3\ldots a^nb^{k-1}a^nb^{k}.
\end{align*}
Observe that if a labeling of a subsegment $I$ appears in $I_{n,k}$ twice (in different places) then the length 
of $I$ is at most 
\begin{align}
\label{f:1}
(k-2)+n+(k-1),
\end{align}
 which is the length of the sequence $b^{k-2}a^nb^{k-1}$.
If a labeling of a segment $I$ appears in two labelings $I_{n,k}$ and $I_{n',k'}$, with $n\neq n'$, then
its length is at most  
\begin{align}
\label{f:2}
\min \{2n+k-1, 2n'+k'-1\},
\end{align}
that corresponds to the sequence $a^nb^{k-1}a^n$ or $a^{n'}b^{k'-1}a^{n'}$.
Combining (\ref{f:1}) and (\ref{f:2}) we have that if a subsegment $I$ of $I_{n,k}$ appears in two different
places in the family $\{ I_{n,k}\; | \; n\neq n' \; \rm{if}\; (n,k)\neq (n',k') \}$ then its length is less than
\begin{align}
\label{f:3}
2(k+n).
\end{align}

The following technical claim follows from an elementary calculation.

\medskip

\noindent 
{\bf Claim.} $\forall E,N\in \mathbb N$ there exist $E$ pairwise distinct numbers 
$n_1,n_2,\ldots, n_E\geqslant N$ and  
$E$ numbers  $k_1,k_2,\ldots,k_E\geqslant N$ such that $|I_{n_i,k_i}|=|I_{n_j,k_j}|$, for all $i,j$.
\medskip

\begin{figure}[h!]
\centering
\includegraphics[width=0.6\textwidth]{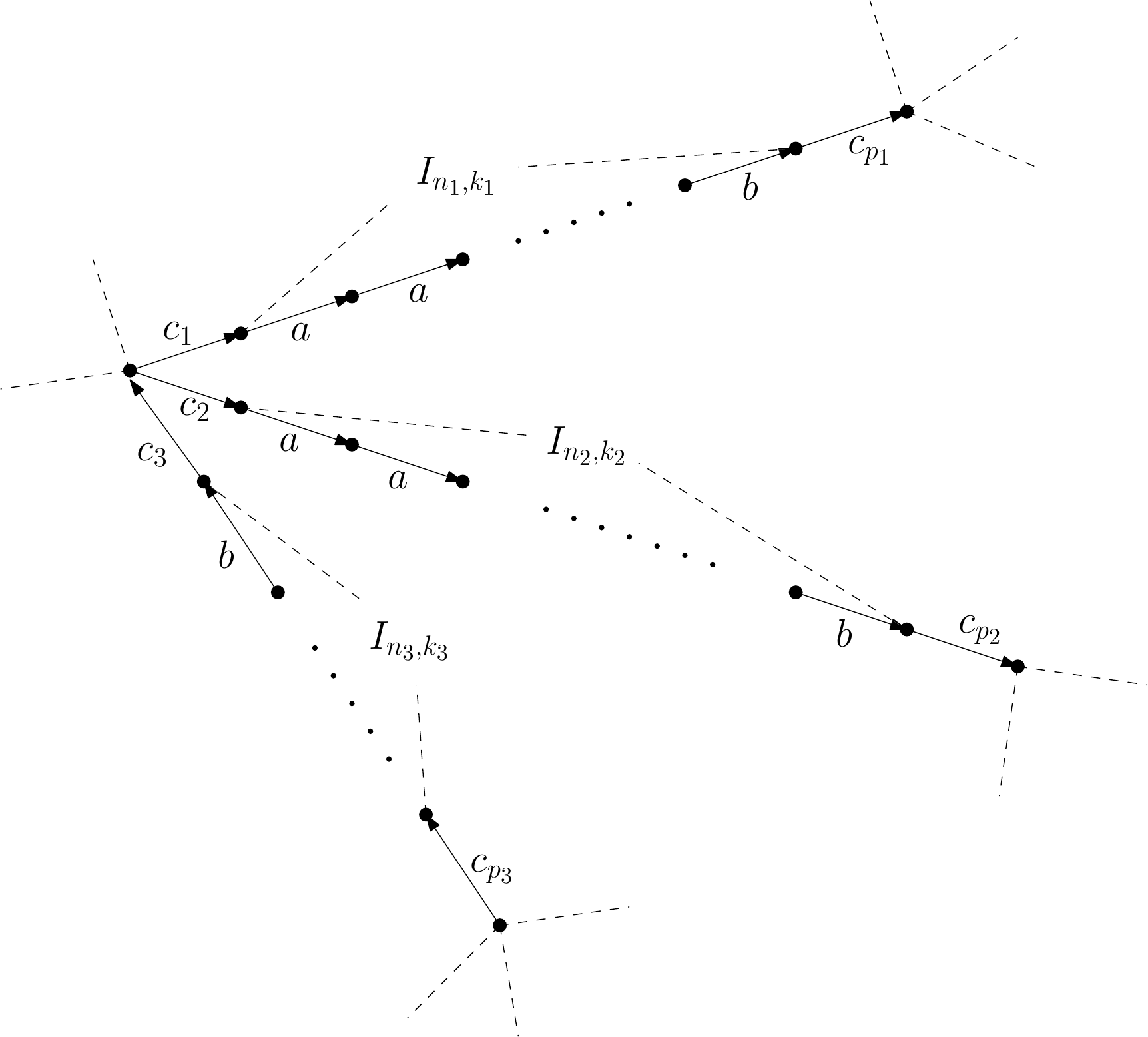}
\caption{Labeling of $\Gamma_i^{j_i}$.}
\label{f:app}
\end{figure}

Now we come to the actual subdivision and labeling of $(\Gamma_i)_{i\in \mathbb N}$. We proceed inductively.
Let $c_1,\ldots, c_d$ be pairwise different and different from $a,b$, letters.
We label $\Gamma_i^{j_i}$ such that edges in the subdivision adjacent to an original vertex of $\Gamma_i$ are all labeled
by different letters among $c_1,\ldots, c_d$ (this is to prevent foldings). Then we use labelings
$I_{n,k}$ to label further subdivided edges of each $\Gamma_i$ --- see Figure~\ref{f:app}. 
Assume that
we defined the required numbers $j_1,\ldots,j_M$, and we found the required labeling of the subfamily $(\Gamma_i^{j_i})_{i\leqslant M}$.
%
%
We find now $j_{m+1}$ and the required labeling of $(\Gamma_i^{j_i})_{i\leqslant M+1}$, that is, we show
an appropriate labeling of $\Gamma_{M+1}^{j_{M+1}}$.
Let $N$ be greater than any index $n$ appearing in $I_{n,k}$ used for labeling the family $(\Gamma_i^{j_i})_{i\leqslant M}$, and greater
than $8/\lambda_{M+1}$.

Let $E$ be the number of edges in $\Gamma_{M+1}$. Let $n_1,\ldots,n_E$, and $k_1,\ldots,k_E$ be as in Claim (for $E,N$ as specified
above).
Then we set $j_{M+1}:=|I_{n_i,k_i}|+2$ and we label each of $E$ subdivided edges of $\Gamma_{M+1}$, using one
of $I_{n_i,k_i}$ as:
\begin{align*}
c_ma^{n_i}ba^{n_i}b^2a^{n_i}b^3\ldots a^{n_i}b^{k_i-1}a^{n_i}b^{k_i}c_p.
\end{align*}

We check now that $(\Gamma_i^{j_i})_{i\leqslant M+1}$ satisfies the small cancellation condition. Let $p$ be a path in some
$\Gamma_i^{j_i}$ such that $p$ appears also elsewhere --- in the same graph or in another one. 
Then there is an edge of $\Gamma_i$ such that the intersection of $p$ with the $j_i$--subdivision $e$ of this edge has length at least
$|p|/2$. If for the labeling of $e$ we used the labeling $I_{n,k}$ (with $j_i=|I_{n,k}|+2$) then we have a subsegment 
of $I_{n,k}$ of length
at least $|p|/2-1$ appearing also in another 
place in the family $\{ I_{n,k}\; | \; n\neq n' \; \rm{if}\; (n,k)\neq (n',k') \}$.
By the formula (\ref{f:3}), it means that
\begin{align*}
|p|/2-1 < 2(k+n).
\end{align*}
Since $j_i=|I_{n,k}|+2=kn+k^2/2+k/2+2$ we obtain
\begin{align}
\label{e:6.5}
\begin{split}
|p|&< 4(k+n)+2=\frac{4(k+n)+2}{j_i}\cdot j_i=\frac{4(k+n)+2}{kn+k^2/2+k/2+2}\cdot j_i\\&
\\&=\frac{8}{k}\cdot\frac{(k^2+nk)/2+k/4}{kn+k^2/2+k/2+2}\cdot j_i < \frac{8}{k}\cdot j_i<\lambda_i \cdot j_i<\lambda_i \cdot \gi\,\Gamma_i^{j_i},
\end{split}
\end{align}
since $k>8/\lambda_i$. This proves the small cancellation condition for $(\Gamma_i^{j_i})_{i\leqslant M+1}$ and, by induction, finishes the proof of the theorem.
\end{proof}

Using the same method one can prove the following stronger result that will be used for providing examples
in Section~\ref{s:ex}.

\begin{tw}[Lacunary labeling of subdivisions]
\label{p:scsubl}
Let $\ov \Gamma = (\Gamma_i)_{i\in \mathbb N}$ be an infinite family of finite graphs with degree bounded by $d$ and
girth tending to infinity as $i\to\infty$.
For every $\ov \lambda = (\lambda_i)_{i\in \mathbb N}$ with $\lambda_i\in (0,1)$, 
there is $n_0>0$ and there exists a sequence $(j_i)_{i \geqslant n_0}$ of natural numbers with the following property.
There exists a labeling of $(\Gamma_i^{j_i})_{i \geqslant n_0}$ by $d+2$ letters satisfying the $C'(\ov \lambda)$--small cancellation condition and, moreover,  satisfying the lacunarity condition of Definition~\ref{d:cond}, for a given $K>0$.
\end{tw}
\begin{proof}
We use the same labeling of subdivisions as in the proof of Theorem~\ref{p:scsub}. We follow the notations of that proof.
For simplicity, by $\Gamma$ we denote $\Gamma_i$, for a given $i$, and by $\Gamma'$ we denote its $j_i$--subdivision $\Gamma_i^{j_i}$.
By the formula (\ref{e:6.5}) the labeling of $\Gamma'$ satisfies the $C'(\lambda)$--small cancellation condition for $\lambda=\frac{4(k+n)+2}{kn+k^2/2+k/2+2}$.
Observe that 
\begin{align*}
\gi\,\Gamma'=j_i\cdot \gi\, \Gamma \;\;\;\;\;\;\mr{and} \;\;\;\;\;\; b_{\Gamma'}(t)\leqslant j_i \cdot b_{\Gamma}(t/j_i),
\end{align*}
where $b_{\Gamma},b_{\Gamma'}$ are the functions introduced in Definition~\ref{d:cond}. Therefore, we obtain the following
\begin{align*}
b_{\Gamma'}(\lambda\cdot \gi\,\Gamma')&\leqslant j_i \cdot b_{\Gamma}(\lambda \cdot j_i \cdot \gi\, \Gamma /j_i)\\&=
j_i\cdot \gi\, \Gamma \cdot \frac{b_{\Gamma}(\lambda \cdot \gi\, \Gamma)}{\gi\, \Gamma}=\frac{b_{\Gamma}(\lambda \cdot \gi\, \Gamma)}{\gi\, \Gamma}\cdot \gi\, \Gamma'.
\end{align*}
For sufficiently large $j_i$, that is, for large $k,n$, we have that $\lambda$ is small. For large $i$ the girth is arbitrarily big.
Thus, setting $n_0$ and $j_i$ (for $i\geqslant n_0$) large enough we can obtain
$\frac{b_{\Gamma}(\lambda \cdot \gi\, \Gamma)}{\gi\, \Gamma}\leqslant K$. Hence the lacunarity condition is satisfied.
\end{proof}
\medskip

In \cite{OllivierWise} it is shown that for a finite family of graphs, random labeling leads to a small cancellation
 labeling of some subdivided graph (up to folding). Their method does not extend to infinite families. 
 A random labeling of an infinite expander family, used by Gromov~\cite{Gro,AD} in his construction of Gromov's monster,
satisfies the so-called \emph{geometric} small cancellation condition (which is of rather different nature).
Thus, it is very interesting to know
whether a random labeling of some subdivisions of an infinite family of bounded degree graphs does satisfy the $C'(\ov \lambda)$--small cancellation condition.

\section{Examples}
\label{s:ex}

In this section, we first give examples of infinite graphical small cancellation presentations that do not reduce to classical
small cancellation presentations treated in \cites{W-sc, AO}. Then we show that the construction of $\mathbb{Z}/2$--homology covers 
from \cite{AGS} produces a sequence of regular graphs satisfying the $(\beta,\delta)$--separation property required
by the lacunary walling condition. Such a sequence is a natural candidate for a sequence $(r_i)_{i\in \mathbb N}$  satisfying the hypothesis of Theorem~\ref{t:nota}, and, hence,
for the construction of a group which is coarsely embeddable into a Hilbert space but not coarsely amenable.
 
\medskip

\noindent
{\bf Example 1.} For any natural number $k$, let $\Theta_k$ be the graph being the union of three segments: $I_a,I_b,I_c$ of length $3k$ each, with three start-points identified and three end-points identified --- see Figure~\ref{f:ex1}.
We label edges of $I_a,I_b,I_c$ in order as, respectively:
\begin{align*}
a_1a_2\ldots a_ke_1e_2\ldots e_ka_1'a_2'\ldots a_k',\\
b_1b_2\ldots b_ke_1'e_2'\ldots e_k'b_1'b_2'\ldots b_k',\\
c_1c_2\ldots c_ke_1''e_2''\ldots e_k''c_1'c_2'\ldots c_k'.
\end{align*}  
Using this labeling we define walls. There are four kinds of them, for every $i=1,\ldots,k$ --- see Figure~\ref{f:ex1}:
\begin{itemize}
\item
the wall $w_{e_i}$ consists of edges: $e_i,e_i',e_i''$;
\item
the wall $w_{a_i}$ consists of edges: $a_i,a_i'$;
\item
the wall $w_{b_i}$ consists of edges: $b_i,b_i'$;
\item
the wall $w_{c_i}$ consists of edges: $c_i,c_i'$.
\end{itemize}   
It is easy to observe that the 
following inequalities hold:
\begin{align*}
d(e_i,e_i')+1 = d(e_i',e_i'')+1 = d(e_i'',e_i)+1\geqslant 2k+1>\frac{1}{3}\cdot 6k=\frac{1}{3}\cdot \gi\,\Theta_k  ,\\
d(a_i,a_i')+1=d(b_i,b_i')+1=d(c_i,c_i')+1=2k =\frac{1}{3}\cdot \gi\,\Theta_k.
\end{align*}  
Moreover, for every geodesic the fraction of edges meeting this geodesic twice is at most $2/3$.
Thus, the system of walls satisfies the \bds property for $\beta = 1/3$ and $\delta =2/3$.
\medskip

Using Theorem~\ref{p:scsubl}, one can now find a sequence $(k_j)_{j\in \mathbb N}$ of natural numbers 
and an appropriate explicit labeling
(different from the one used just for defining walls)
of graphs $\Theta_{k_j}$ that defines an infinite graphical small cancellation presentation of a group
satisfying the lacunary walling condition.

\begin{figure}[h!]
\centering
\includegraphics[width=0.5\textwidth]{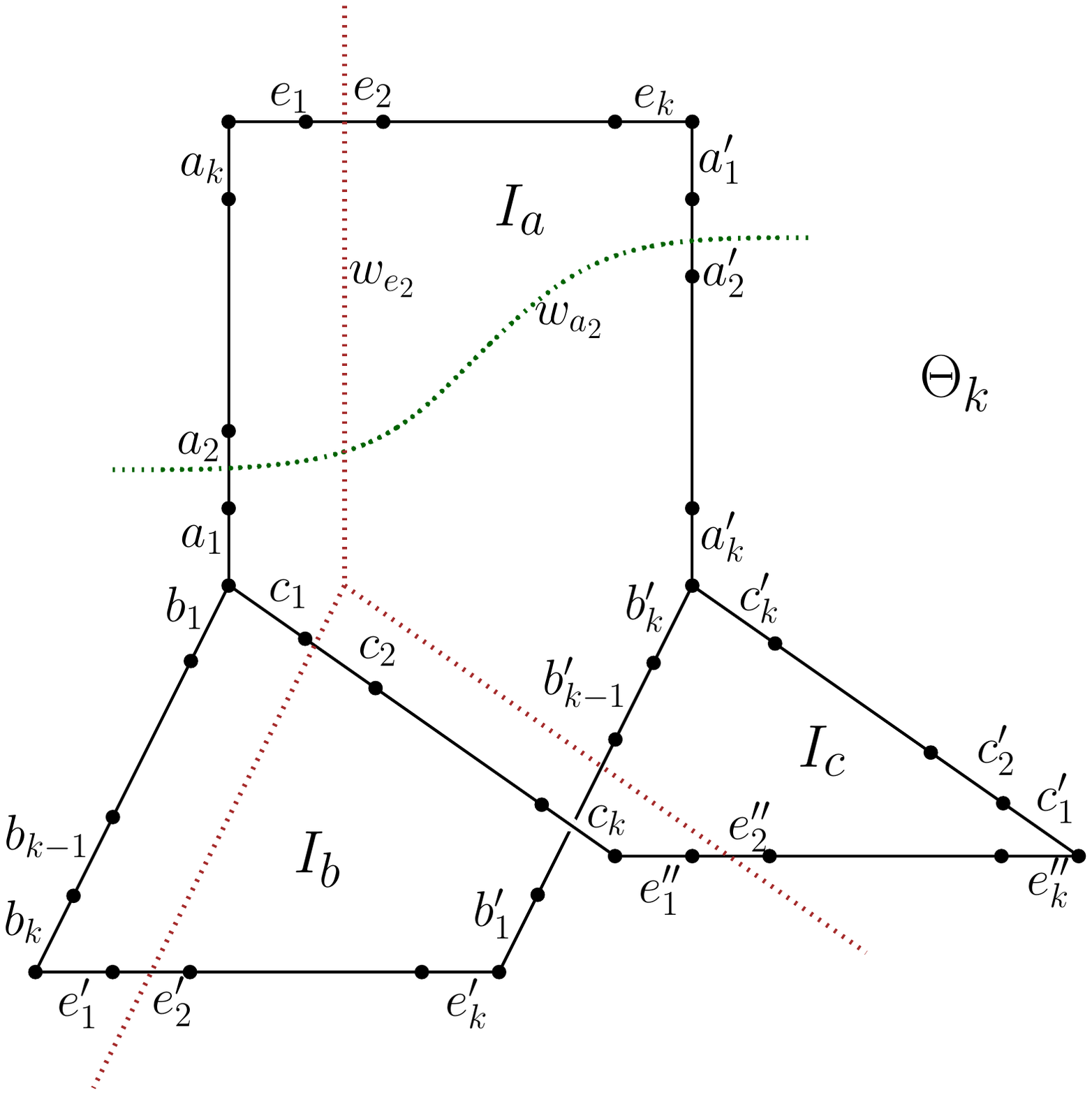}
\caption{Example 1.}
\label{f:ex1}
\end{figure}

\medskip

\noindent
{\bf Example 2.} Let $I^n$ be the $1$--skeleton of an $n$--cube. Subdividing every its edge into $k$ edges we obtain
the graph $I^n_k$ --- see Figure~\ref{f:ex2}. This graph possesses a natural wall structure --- opposite edges in every $4k$ cycle
belong to the same wall (edges $e,e',e'',e'''$ form the wall $w$ in Figure~\ref{f:ex2}).
For any median graph $r$ (i.e.\ the $1$--skeleton of a CAT(0) cube complex) one can equip each $k$--subdivision of $r$ (every edge subdivided into $k$ edges) with a wall system,
applying the above rule to each cube (or simpler -- just to each square) of $r$.
Observe that for any two edges $e,e'$ of the same wall, we have
\begin{align*}
d(e,e')+1\geqslant 2k,
\end{align*}  
which means that the $\beta$--condition from Definition~\ref{d:sep} is satisfied for $\beta = 1/2$.
Furthermore, one easily sees that the
$\delta$--condition from Definition~\ref{d:sep} is satisfied for $\delta = 1/3$.

\begin{figure}[h!]
\centering
\includegraphics[width=0.4\textwidth]{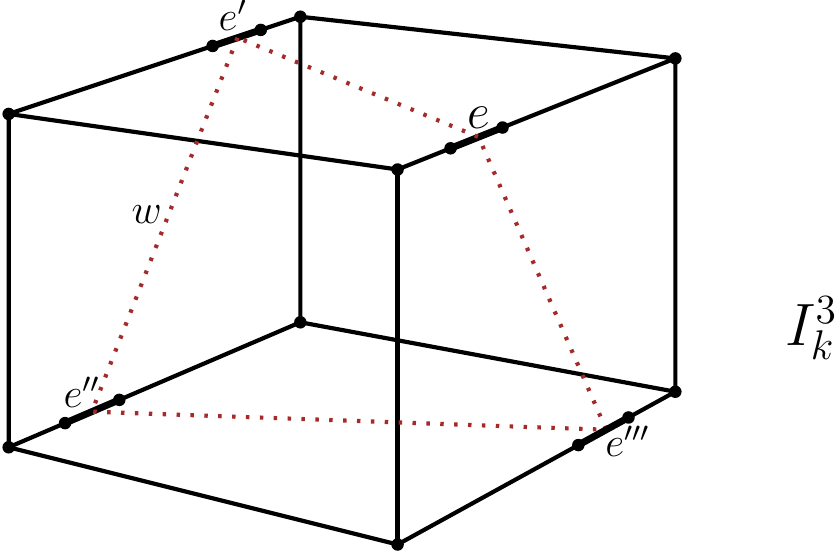}
\caption{Example 2.}
\label{f:ex2}
\end{figure}

Using Theorem~\ref{p:scsubl}, any sequence of median graphs $(r_i)_{i\in \mathbb N}$ may be transformed (up to taking a subsequence) into a sequence of their subdivisions $(r_i')_{i\in \mathbb N}$ with a small
cancellation labeling, satisfying the lacunary walling condition.
The resulting group acts then properly on the corresponding space with walls and hence, possesses the Haagerup property.\medskip

\noindent
{\bf Example 3.}  We recall the construction of the box space which is coarsely embeddable into a Hilbert space but not coarsely amenable~\cite{AGS}.
Let $\mathbb F_m$ be the free group of rank $m\geqslant 2$. Arzhantseva-Guentner-\v{S}pakula's box space is a regular graph 
$$\Theta= \bigsqcup^{\infty}_{n=1}\Theta_n:=\bigsqcup^{\infty}_{n=1} \mathbb{F}_m/((\mathbb{F}_m^{(2)})^{(2)})^{\ldots (2)},$$
which is the disjoint union of the Cayley graphs $\Theta_n$ of quotients of $\mathbb F_m$  by the subgroups generated iteratively, over $n\geqslant 1$, 
by the squares of the group elements. That is, $\mathbb{F}_m^{(2)}$ is the normal subgroup of $\mathbb{F}_m$ generated by all the squares of elements
of $\mathbb{F}_m$, then $\displaystyle{(\mathbb{F}_m^{(2)})^{(2)}}$ is such a subgroup of $\mathbb{F}_m^{(2)}$, and so on. The corresponding Cayley graphs are viewed with respect to the 
canonical image of the free generators of $\mathbb F_m$.

The graph 
$\Theta_n$ is the $\mathbb{Z}/2$--homology cover of  $\Theta_{n-1}$ (with  $\Theta_0$ being the bouquet of $m$ circuits), that is, a regular cover of $\Theta_{n-1}$ whose the group of deck transformations is the\
${\rm rank}(\pi_1(\Theta_{n-1}))$--fold direct sum of $\mathbb{Z}/2\mathbb{Z}$'s.  

The graph $\Theta$ is the graph with walls~\cite{AGS}*{Section 3}. For each edge $e\in E(\Theta_{n-1})$ and the covering map $\pi_n\colon \Theta_n\to \Theta_{n-1},$
 the wall $w_e$ is defined by $w_e:=\pi^{-1}_n(e)\subseteq E(\Theta_n)$ and  $\{ w_e \mid e \in E(\Theta_{n-1})\}$ provides the wall structure on $\Theta_n$, for all $n\geqslant 1$.

The following general observation shows that this wall structure on $\Theta$ does satisfy the $\beta$--condition from Definition~\ref{d:sep}. 

\begin{lem}\label{lem:bbox}
Let $\pi\colon \wt \Gamma \to \Gamma$ be a $\mathbb Z/2$--homology cover endowed with the wall structure as above.
Then the $\beta$--condition holds, with $\beta=1/2.$
\end{lem}
\begin{proof}
Observe that  $\gi \, \wt \Gamma = 2 \cdot \gi \, \Gamma$.
Indeed, let $\gamma$ be a cycle in $\Gamma$ of length $\gi \, \Gamma$. We remove an arbitrary edge from $\gamma$, include the remaining path
into a spanning tree of $\Gamma$, and take the $\mathbb Z/2$--homology cover of $\Gamma$ with respect to this spanning tree. 
The resulting cover coincides with $\wt \Gamma$, as it does not depend on the choice of a spanning tree, and we have $\gi \, \wt \Gamma = 2 \cdot \gi \, \Gamma$. 

Let $e,e'$ belong to a common wall in $\wt \Gamma$. Then $\pi(e)=\pi(e')$, by our definition of walls~\cite{AGS}. Let $\gamma \subseteq \wt \Gamma$ be a geodesic between
$e$ and $e'$, with endpoints $v\in e$ and $v'\in e'$. Let $w\neq v$ be another vertex of $e$. If
$\pi(v)=\pi(v')$ then $\pi(\gamma)$ contains a closed path in $\Gamma$. If $\pi(w)=\pi(v')$ then there is a closed path in $\Gamma$ of length
at most $d(v,v')+1$.
In both cases we obtain
\begin{align*}
d(e,e')+1=d(v,v')+1\geqslant \gi \, \Gamma = (1/2)\gi \, \wt \Gamma.
\end{align*}
\end{proof}

In order to guarantee the $\delta$--condition from Definition~\ref{d:sep}, we take  the $\mathbb{Z}/2$--homology cover of an appropriately chosen sequence of graphs
(instead of the above tower of successive covers starting with the bouquet of $m$ circuits).

Let $\Lambda = (\Lambda_i)_{i\in \mathbb N}$ be an infinite family of finite 2-connected graphs, with 
$\gi\, \Lambda_i\to\infty$ as $i\to\infty$, and such that $\di\, \Lambda_i/\gi\, \Lambda_i\leqslant M$ for some $M>0$, uniformly over $i\in \mathbb N$.
Let $\wt\Lambda = (\wt\Lambda_i)_{i\in \mathbb N}$ be the corresponding infinite family of the $\mathbb Z/2$--homology covers:
$\wt\Lambda_i$ is the  $\mathbb Z/2$--homology cover of $\Lambda_i$, for each $i\in \mathbb N$. Observe that $\gi\, \wt\Lambda_i\to\infty$ as $i\to\infty$.
We endow $\wt\Lambda$ with our wall structure as above.

\begin{lem}\label{lem:dbox}
Let $\gamma$ be a geodesic in $\wt\Lambda$.  Then there exists $\delta\in (0,1)$ such that
 the number of edges in $\gamma$ whose walls meet $\gamma$ at
least twice is at most $\delta |\gamma|$. 
\end{lem}

\begin{proof}
 The image of $\gamma$ under the covering projection $\pi\colon \wt\Lambda \to\Lambda$ is a so-called admissible path $\pi(\gamma)$ in $\Lambda$, 
 see~\cite{AGS}*{Definition 3.5}. We have $|\gamma|=|\pi(\gamma)|$ for the edge-length, see~\cite{AGS}*{Lemma 3.6 and Proposition 3.8}, and the path $\pi(\gamma)$ 
 either does not contain any loop, or else every edge on any loop it contains is traversed exactly once, see~\cite{AGS}*{Lemma 3.12}. 
 Note that $\pi(\gamma)$ has no any backtrack since $\gamma$ is geodesic, see~\cite{AGS}*{Remark 3.9}.
 Therefore, if it does not contain any loop, then no walls meet $\gamma$ at least twice, and the $\delta$--condition
 is satisfied for all $\delta\in (0,1)$. 
 
 Suppose now that $\pi(\gamma)$ contains $L\geqslant 1$ loops. Then $\pi(\gamma)$ decomposes
 into $L$ loops $c_1,\ldots, c_L$ traversed exactly once, $l$ geodesics $t_1,\ldots, t_{l}$ which are followed twice (toward and backward each loop as well as between such ``flowers''), and
 finitely many geodesics $s_1, s_2, \ldots$, traversed exactly once, joining several trees  formed by $t_j$'s, see Figure~\ref{f:adm}. 
 
\begin{figure}[h!]
\centering
\includegraphics[width=1\textwidth]{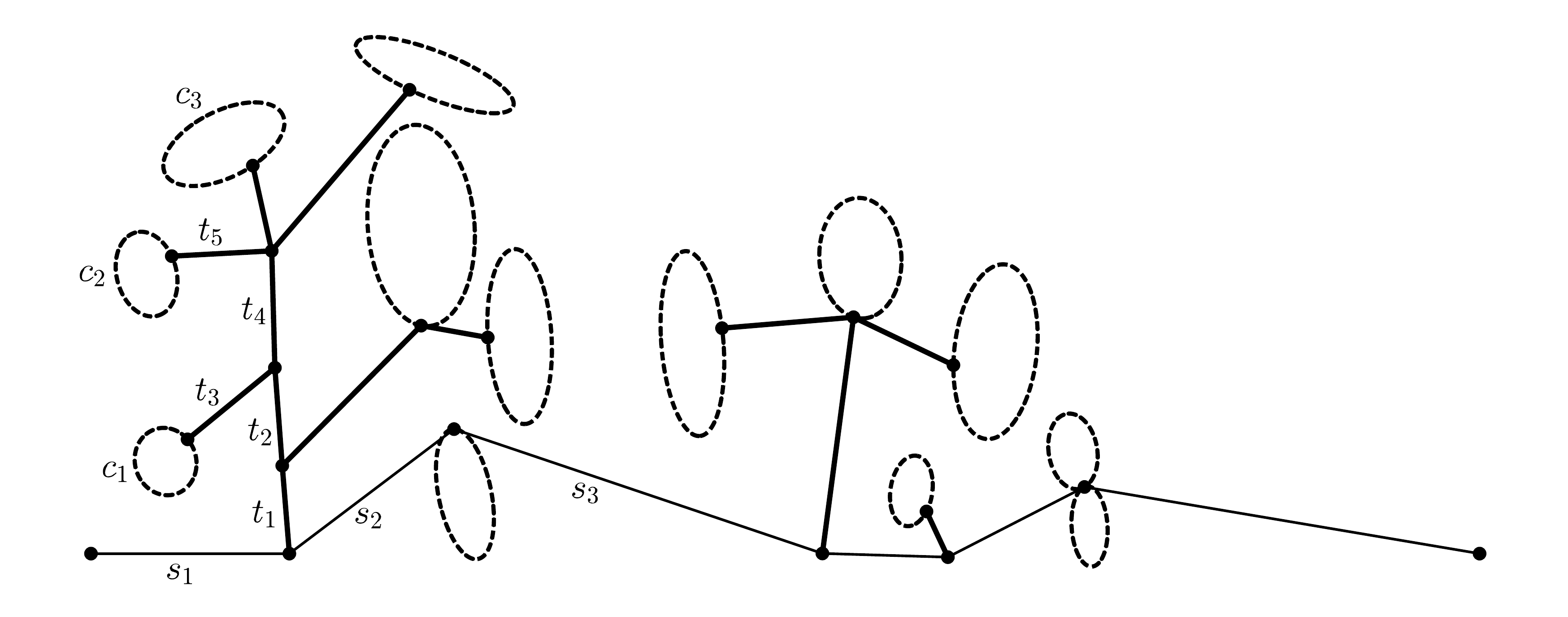}
\caption{Decomposition of a shortest admissible path $\pi(\gamma)$ in $\Lambda$.}
\label{f:adm}
\end{figure}

 For all $k$ and $j$, we have $|c_k|\geqslant \gi\, \Lambda_i$ and $|t_j|\leqslant \di\, \Lambda_i$, for some $i$.  
 Using the induction on $L$, taking into account the tree structures formed by $t_j$'s, one checks that $l\leqslant 2L$. Thus, we have
 $$
 |\gamma|/|\gamma|_{\mathcal W}=\frac{\sum_k|c_k|+2\sum_j|t_j| +\sum_q|s_q|}{\sum_k|c_k|+\sum_q|s_q|}\leqslant 2+ \frac{2(2L)\di\,\Lambda_i}{L\,\gi\,\Lambda_i}\leqslant 2+4M,
 $$ 
 where $|\gamma|_{\mathcal W}$ is the  length of $\gamma$ with respect to the wall pseudo-metric, 
 that is, the number of edges in $\gamma$ whose walls meet $\gamma$ exactly once~\cite{AGS}*{Proposition 3.10}.
Therefore, $\wt\Lambda$ satisfies the $\delta$--condition with $\delta=1-\frac{1}{2+4M}=\frac{1+4M}{2+4M}$.
\end{proof}

We have just checked, using the results of~\cite{AGS}, that the graph metric and 
the wall pseudo-metric on $\wt\Lambda$ are bi-Lipschitz equivalent.  This is the main result of~\cite{Ostrov} 
(also giving $\delta$, see inequality (3) in that paper), where
such a choice of $\Lambda$ was explored, in the context of metric embeddings into the Banach space $\ell_1$.

The existence of infinite families $\Lambda = (\Lambda_i)_{i\in \mathbb N}$ as above, consisting of finite regular graphs of vertex degree at least 3, is well-known. For instance, the famous Ramanujan
sequence of Lubotzky-Phillips-Sarnak provides such a family of finite $(p+1)$--regular graphs, where $p$ is an odd prime.

Thus,  $\wt\Lambda = (\wt\Lambda_i)_{i\in \mathbb N}$ is coarsely embeddable into a Hilbert space,
has a bi-Lipschitz embedding into the Banach space $\ell_1$ (as every wall space has an isometric embedding into $\ell_1$),
 but is not coarsely amenable since $\gi\, \wt\Lambda_i\to\infty$ as $i\to\infty$ and the graphs are regular of vertex degree at least 3~\cite{Willett}.
 By Lemmas~\ref{lem:bbox} and~\ref{lem:dbox}, $\wt\Lambda = (\wt\Lambda_i)_{i\in \mathbb N}$ satisfies the $(\beta,\delta)$--condition required
 by the lacunary walling condition.

\section{Discussion on the optimality}
\label{optim}

In this section, we argue that our lacunary walling condition is essential for obtaining the main result
--- linear separation property for small cancellation complexes; see Theorem~\ref{t:linear} in Introduction and
Theorem~\ref{p:lsp}. We focus on the \bds property and on the lacunarity (together with the compatibility)
from Definition~\ref{d:cond}. We provide examples showing that if any of them fails then the wall pseudo-metric
may be non-comparable with the $1$--skeleton metric. 

\subsection{\texorpdfstring{\bds}{(beta, delta)--separation} property.}
\label{s:bds}

The $\beta$--condition from Definition~\ref{d:sep} is essential for the definition of walls in $X$; see Section~\ref{s:walls}.
Without it one cannot usually extend the walls in relators to the whole $X$. It corresponds to the \emph{$\frac{\pi}{2}$--strong separation property} from \cite[Section 5.k]{W-qch}, and 
is essential also in the finitely presented case. For the rest of this subsection
we therefore focus on the $\delta$--condition from Definition~\ref{d:sep}. This condition is required in the infinitely presented case
and may be easily omitted in the finitely presented one.
\medskip  

Consider the $(4k+2)$--gon $C_k=(v_1,v_2,\ldots,v_{4k+2},v_{4k+3}=v_1)$.
Assume that for every $i\in \{1,2,\ldots,k\} \cup \{2k+2,2k+3,\ldots,3k+1 \}$
the edges $v_iv_{i+1}$ and $v_{i+k+1}v_{i+k+2}$
are dual to common wall $w_i$. Moreover, let the edges $v_{k+1}v_{k+2}$ and
$v_{3k+2},v_{3k+3}$ be dual to common wall $w_{k+1}$ --- see Figure \ref{f:14}.
Observe that the system of walls satisfies the $\beta$--condition from Definition~\ref{d:sep}, for $\beta$ approaching
from above $1/4$ while $k$ grows.
Note however, that
$d(v_1,v_{2k+2})=2k+1$ while $d_W(v_1,v_{2k+2})=1$, where $d_W$ is the wall pseudo-metric
given by walls $w_i$,  that is, the number of walls separating given vertices.
This means that an infinite family $\{ C_{k_i} \}_{i\in \mathbb N}$ does not satisfy the $\delta$--condition
for any $\delta >0$.

\begin{figure}[h!]
\centering
\includegraphics[width=0.5\textwidth]{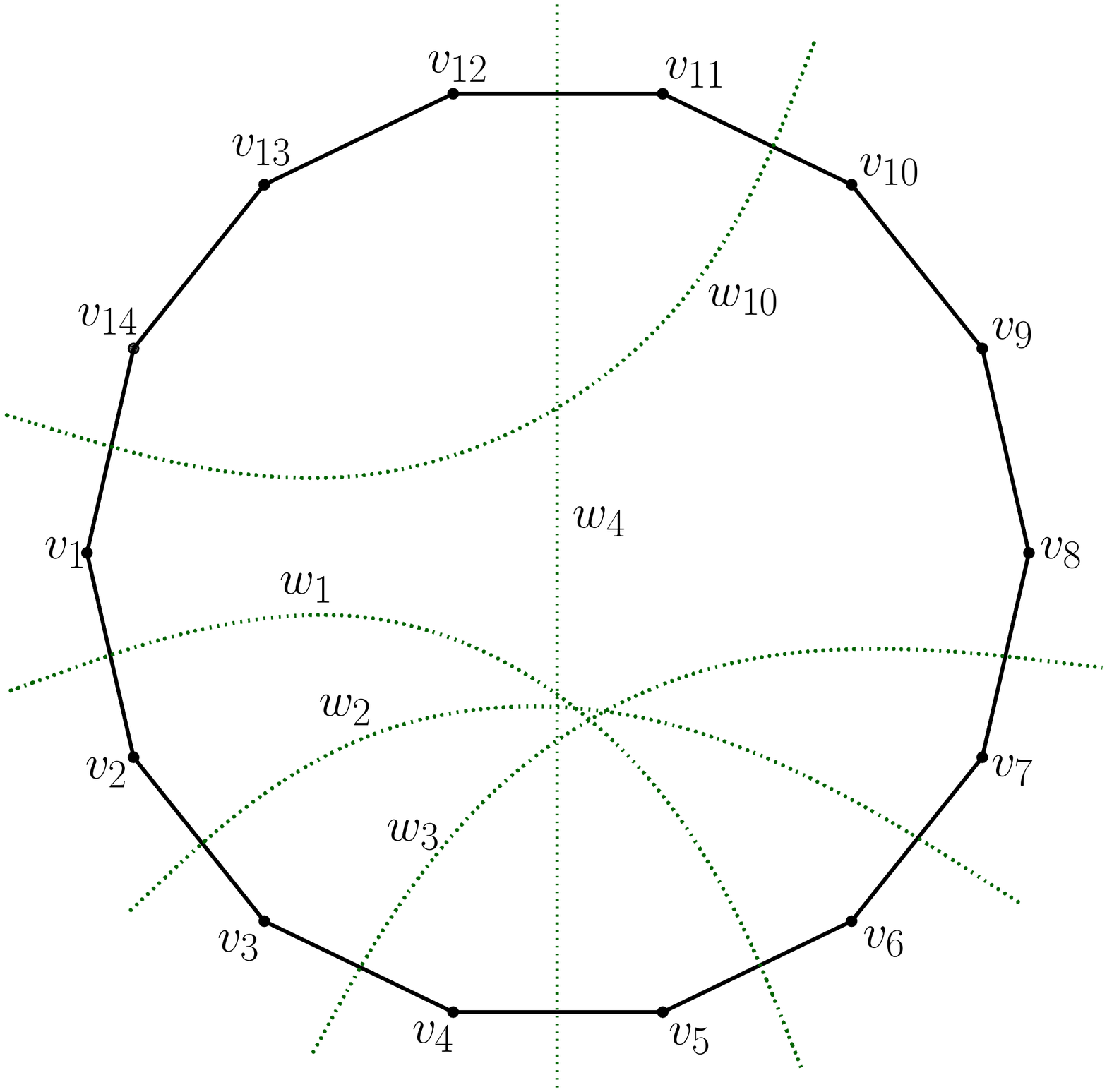}
\caption{Failure of the $\delta$--condition.}
\label{f:14}
\end{figure}

One can easily construct a (classical) infinite small cancellation presentation with relators being 
$(4k+2)$--gons as above. When equipped with walls induced by the above walls in relators, the Cayley complex 
of the group $G$ defined by such a presentation becomes a space with walls.
Yet $X$ contains arbitrarily long geodesic (contained moreover in relators) separated by one wall. Consequently, the group $G$ does not act properly on the obtained space with walls. 
It shows that both: the $\beta$--condition and the $\delta$--condition, are necessary in our approach.

Note that the system of walls as above can satisfy (choosing an appropriate small cancellation labeling) the 
\emph{$\frac{\pi}{2}$--strong separation property} from \cite[Section 5.k]{W-qch}. This means that the pathologies as just described
are characteristic for infinite presentations, and not for finite ones.

\subsection{Lacunarity/Compatibility}
\label{s:lacun}

Consider a relator $r$ as in Figure~\ref{f:lacun1}. Here, walls correspond to the sets of edges: 
$$\{ a_i,a_i',a_i'',a_i'''\}, \{ b_i,b_i',b_i'',b_i'''\},\hbox{ for }
i=1,2,3,4, \hbox{ and } \{c,c' \},\{d,d'\},\{e,e'\},\{f,f'\}.$$
Observe that $\gi \, r = 16 $, and that the wall system satisfies the \bds property for $\beta=1/4$ and some $\delta <1$. 

\begin{figure}[h!]
\centering
\includegraphics[width=0.8\textwidth]{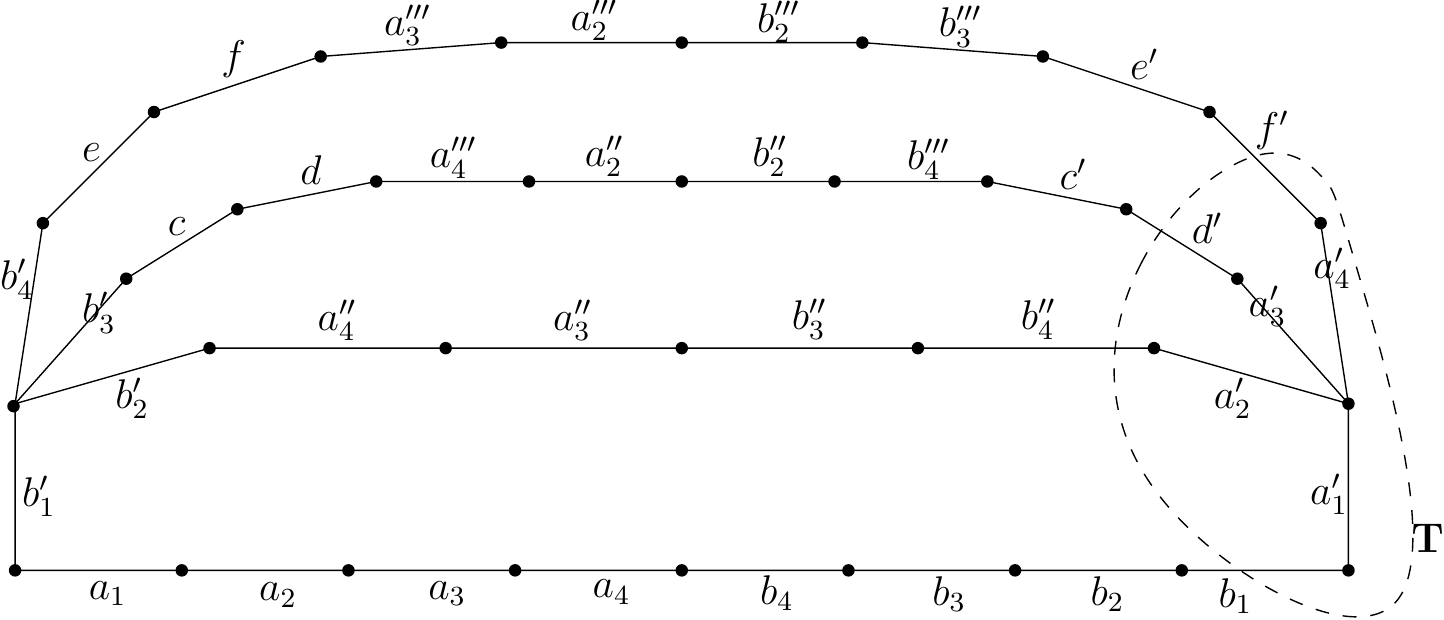}
\caption{Relator $r$.}
\label{f:lacun1}
\end{figure}

For $k=3,4,\ldots,$ we now construct a complex $X_k$ satisfying the $C'(1/8)$--small cancellation condition using the relator $r$ 
--- see Figure~\ref{f:lacun2}.
The complex $X_k$ is the union of $k$ copies of $r$, such that consecutive copies share a common tree of valence $4$ and diameter
$3$, consisting of edges labeled by $a_i'$ or $b_i'$ in $r$ (like the tree $\bf T$ in Figure~\ref{f:lacun1}).

\begin{figure}[h!]
\centering
\includegraphics[width=0.8\textwidth]{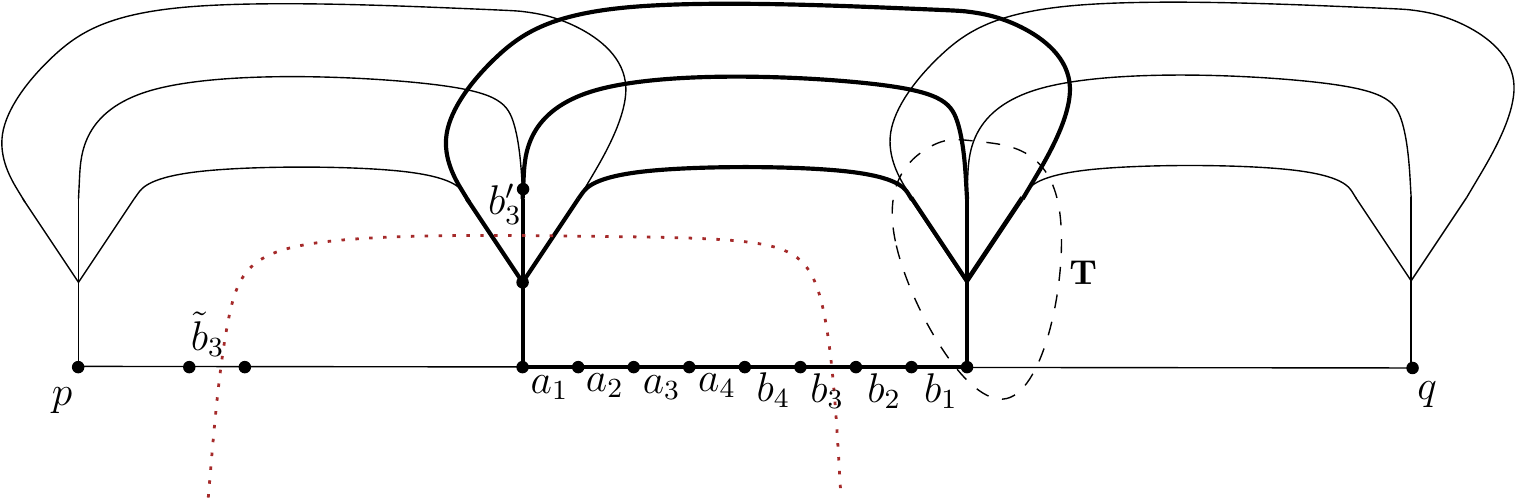}
\caption{The complex $X_k$, for $k=3$.}
\label{f:lacun2}
\end{figure}

We equip the complex $X_k$ with the wall system $\mathcal W$ as in Subsection~\ref{s:walls}, by extending the walls in $r$.
Note that $X_k$ does not satisfy the lacunarity/compatibility conditions from Definition~\ref{d:cond}.
By lacunarity we have to have (using the notations from Definintion~\ref{d:cond}):
\begin{align*}
b_{r}\left( \frac{1}{8}\cdot 16\right) = 4 \leqslant K \cdot 16,
\end{align*}
which implies $K\geqslant 1/4$. For such $K$ however there is no way to satisfy the compatibility condition for any $M> 0$.

Consider $X_3$, and a geodesic with endpoints $p,q$, as in Figure~\ref{f:lacun2}. Observe that no wall in $\mathcal W$, corresponding to $a_i,b_i$ separates $p$ from $q$.
Therefore we have
$\dw (p,q) = 8$ and $d(p,q)=24$, in $X_3$. Similarly, for any $k$, in $X_k$ one may find a geodesic (corresponding to $pq$) of length 
$8k$ whose endpoints are separated by $8$ walls from $\mathcal W$. This shows that in the corresponding infinite union $X_{\infty}$ 
the wall pseudo-metric is not proper.

Similarly, one can construct other examples for arbitrary small cancellation constant $\lambda>0$.
The point here is that the number of edges in the tree $\bf T$ may be exponentially large compared to its diameter.
Then the corresponding walls may ``exhaust" most of the edges in a given geodesic.

Observe that in the above example we use only finitely many (precisely, one) types of relators. This corresponds to the case 
of finitely presented groups. Wise \cite[Section 5.k]{W-qch} uses a notion of \emph{$\frac{\pi}{2}$--strong separation property}
for relators, to obtain an analogous linear separation for finitely presented graphical small cancellation groups. In our approach, 
the $\frac{\pi}{2}$--strong separation property is replaced by the lacunarity condition. Neither of these properties implies another one.
We decided to use the lacunarity as a condition that suits better our -- quantitative -- approach to spaces with walls.


\end{document}